\documentclass[10pt,draftcls,onecolumn]{IEEEtran}

\IEEEoverridecommandlockouts                              % This command is only needed if 
\usepackage{graphics} % for pdf, bitmapped graphics files
\usepackage{epsfig} % for postscript graphics files
\usepackage{subfig}
\usepackage{amsmath} % assumes amsmath package installed
\usepackage{amssymb}  % assumes amsmath package installed
\usepackage{amsthm}
\usepackage{mathtools}

\newtheorem{theorem}{Theorem}[section]
\newtheorem{lemma}[theorem]{Lemma}
\newtheorem{proposition}[theorem]{Proposition}
\newtheorem{corollary}[theorem]{Corollary}
\newtheorem{remark}[theorem]{Remark}

\title{\LARGE \bf Optimal Sparse Output Feedback Controller Design:\\
  A Rank Constrained Optimization Approach}  
\author{Reza Arastoo$^\dagger$ \and Nader Motee$^\dagger$ \and Mayuresh V. Kothare$^\ddagger$ 
\thanks{$\dagger$ R. Arastoo and N. Motee are with Department  of
  Mechanical Engineering  and  Mechanics,  Packard  Lab.,  Lehigh
  University,  Bethlehem, PA. Email Addresses: {\tt\small
    \{reza.arastoo,nader.motee\}@lehigh.edu}}  
\thanks {$\ddagger$ M. V. Kothare is with the Department of Chemical
  and Biomolecular Engineering, Iaccoca Hall, Lehigh University,
  Bethlehem, PA.
        Email Address: {\tt\small mayuresh.kothare@lehigh.edu}}}

\begin{document}

\maketitle
\thispagestyle{empty}
\pagestyle{plain}

%%%%%%%%%%%%%%%%%%%%%%%%%%%%%%%%%%%%%%%%%%%%%%%%%%%%%%%%%%%%%%%%%%%%%%%%%%%%%%%%
\begin{abstract}
  \noindent We consider the problem of optimal sparse output feedback controller
  synthesis for continuous linear time invariant systems when the
  feedback gain is static and subject to specified structural
  constraints. Introducing an additional term penalizing the number
  of non-zero entries of the feedback gain into the optimization cost
  function, we show that this inherently non-convex problem can be
  equivalently cast as a rank constrained optimization, hence, it is
  an NP-hard problem. We further exploit our rank constrained approach to define a structured output feedback control feasibility test with global convergence property, then, obtain upper/lower bounds for the optimal cost of the sparse output feedback control problem. Moreover, we show that our problem reformulation allows us to incorporate additional implementation constraints, such as norm bounds on the control inputs or system output, by assimilating them
  into the rank constraint. We propose to utilize a version of the Alternating
  Direction Method of Multipliers (ADMM) as an efficient method to
  sub-optimally solve the equivalent rank constrained problem. As a
  special case, we study the problem of designing the sparsest stabilizing output feedback controller, and show that it is, in fact, a
  structured matrix recovery problem where the matrix of interest is
  simultaneously sparse and low rank. Furthermore, we show that this
  matrix recovery problem can be equivalently cast in the form of a
  canonical and well-studied rank minimization problem. We finally
  illustrate performance of our proposed methodology using 
  numerical examples.
\end{abstract}
%%%%%%%%%%%%%%%%%%%%%%%%%%%%%%%%%%%%%%%%%%%%%%%%%%%%%%%%%%%%%%%%%%%%%%%%%%%%%%%%
\section{INTRODUCTION}
The problem of optimal linear quadratic controller design has been
extensively studied for several decades. In conventional control, it
is usually assumed that all measurements are accessible to a
centralized controller, while in large scale interconnected systems
this assumption is not practical, since it is often desirable that
subsystems only communicate with a few neighboring components due to
the high cost and, sometimes, infeasibility of communication
links. Therefore, the need to exploit a particular controller
structure, obtained based on the layout of the system network, seems
undeniable. Furthermore, the traditional controller synthesis methods,
which are closely related to solving the Algebraic Riccati Equation, no
longer work when additional constraints are imposed on the structure
of the controller. 

In general, the problem of designing constant gain feedback
controllers subject to additional constraints is NP-hard
\cite{Blondel:1997}. In recent years, numerous attempts have been made
to provide distributed controller synthesis approaches for different
classes of systems
\cite{Schuler:2011,Polyak:2013,Motee:2009,Lin:2011,Krishna:2013}. Bamieh
\emph{et al}, in \cite{Bamieh:2002,Bamieh:2005}, investigated the
distributed control of spatially invariant systems, then the work in
\cite{Motee:2008} has proved that the solution of Riccati and Lyapunov
equations for systems consisting of Spatially Decaying (SD) operators
has SD property, which lends credibility to the search for controllers
that have access only to local measurements. The design of optimal
state feedback gain in the presence of an \emph{a priori} specified
structure, usually in the form of sparsity patterns, is considered in
\cite{Lin:2011}. In their recent papers, Lavaei \emph{et al.}
\cite{Lavaei:2013,Fazelnia:2014,Madani:2014} cast the problem of
optimal decentralized control for discrete time systems as a rank
constrained optimization problem, developed results on the possible
rank of the resulting feasible set, and introduced several
rank-reducing heuristics as well. Wang \emph{et al.} studied the problem of localized $\text{LQR}$/$\text{LQG}$ control and presented a synthesis algorithm for large-scale localizable systems \cite{Wang:2014,Wang:2015}. Frequency domain approaches to design optimal decentralized controllers are also presented in
\cite{Rotkowitz:2006,Rotkowitz:2011,XinQi:2004}. 

In the design of linear feedback controllers for interconnected
systems, a common desired structure is that the controller matrices
are sparse, which  could correspond to a simpler controller topology
and fewer sensors/actuators. However, fewer measurement/communication
links leads to performance deterioration and sometimes even
instability of the overall system. Therefore, there exists a trade off
between the stability and performance of the system and minimizing the
number of non-zero entries of the feedback gain matrices. The problem
of minimizing the number of nonzero elements of a vector/matrix
subject to a set of constraints in inherently NP-hard and arises in
many fields, such as Compressive Sensing (CS) where the inherent
sparseness of signals is exploited in determining them from relatively
few measurements \cite{Donoho:2006}. Since the advent of Compressive
Sensing, considerable work has been done on the design of compressive
measurement matrices based on different criteria such as sparse signal
support detection and estimation  \cite{DeVore:2007,Zahedi:2013},
sparse signal detection and classification
\cite{Zahedi:2012,Davenport:2010}, etc.  

To alleviate the issues caused by the combinatorial nature of
cardinality functions, several convex/non-convex functions have been
proposed as surrogates for the cardinality functions in optimization
problems. For example, in cases where the optimization constraint is
affine, $\ell_1$-norm , as a convex relaxation of $\ell_0$-norm, has proved to work reliably under certain conditions, namely Restricted Isometry Property (RIP) \cite{Candes:2004, Candes:2005,
  Candes:2008}. Thus, $\ell_1$-norm and its weighted versions have
been extensively used in signal processing and control applications
\cite{Lin:2013,Schuler:2011c,Arastoo:2012}. Non-convex relaxations of
the cardinality function, such as $\ell_q$-quasi-norm ($0<q<1$), have
also received considerable attention recently
\cite{Chartrand:2007,Lai:2011}. In
\cite{Motee:2013,Motee:2014,Motee:2014ACC},  it is shown that, for a
large class of SD systems, the quadratically-optimal feedback
controllers inherit spatial decay property from the dynamics of the
underlying system. Moreover, the authors have proposed a method, based
on new notions of $q$-Banach algebras, by which sparsity and spatial
localization features of the same class can be studied when $q$ is
chosen sufficiently small. 

In the present paper, we consider the problem of optimal sparse
feedback controller synthesis for linear time invariant system, in
which convex constraints are imposed on the structure of the
controller feedback gain. The main contribution of our paper is to
propose a novel approach which allows us to equivalently represent the
intrinsically nonlinear constraints, such as closed loop stability
condition and enforcement of controller structure, with a {\em single}
rank constraint in an otherwise convex optimization program. Having
all non-linearities encapsulated in only one rank constraint allows us
to employ one of several existing algorithms to efficiently solve the
resulting problem.

Our results are distinct from those reported in \cite{Lin:2013}, as we present an alternative formulation which not only solves the
regular sparse controller design problem, but also enables us to solve
the output feedback control problem. Furthermore, integrating various
types of nonlinear system constraints, such as constraints on the
controller matrix and its norms, into the existing rank
constraint can be effortlessly implemented in our approach. It should also be noted that the rank constraint emerging in our approach originates from the positive definiteness of the Lyapunov matrix and the properties of fixed rank matrices, thus the ratio of matrix dimension and its rank does not grow with the size of system. In contrast, the rank one constraint appears in \cite{Madani:2014} results from utilizing the auxiliary variable introduced by self multiplying the vector formed by augmenting the states, inputs, and outputs, hence there exist a linear growth of the ratio of the dimension of the matrix to its rank as the number of variable increases, which is a computational drawback in controller synthesis for large scale systems.

We start by augmenting the $\ell_0$-norm of the feedback gain matrix to the
quadratic cost function of our optimization problem. This additional
term penalizes the extra communication links in the feedback
pathway. We then reformulate it into an equivalent optimization
problem where the non-convex constraints are lumped into a rank
constraint. Based on the notions of holdable ellipsoid, we 
propose a reformulation of the problem to incorporate norm bounds on
the control inputs and outputs of the system, which usually appear in
controller implementations. Employing a convex relaxation of the added
cardinality term, based on the weighted $\ell_1$-norm, we argue that
Alternating Direction Method of Multipliers (ADMM) is well-suited to
solve our problem, since our search is to obtain a solution with
an a priori known rank. ADMM iteratively solves the rank-unconstrained
problem and  projects the solution into the space of the matrices
with the desired rank until the convergence criteria are met.  We
further investigate the special case of designing the sparsest
stabilizing controller, and show that this problem can be rewritten as a
rank minimization problem. Rank minimization problems have received
considerable attention in recent years \cite{Recht:2010, Recht:2011,
  Mesbahi:1997}. In \cite{Recht:2007}, it is shown that if a certain
Restricted Isometry Property holds for the linear transformation
defining the constraints, the minimum rank solution can be recovered
by solving the minimization of the nuclear norm over the feasible
space. Therefore, the nuclear norm may be used as a proxy for the rank
minimization in our problem.

The remainder of this paper is organized as follows. In Section
\ref{prob-formu-sec}, the general optimal sparse output feedback
control problem setup is defined. Section \ref{rank-cons-sec}, we
reformulate the optimal sparse output feedback control problem as a
rank constrained problem, and develop several results based on the
proposed reformulation. In Section \ref{sec:conv_relaxation}, we study
the convex relaxation of this problem, and discuss the application of
ADMM in solving the problem. The special case where the sparsity
penalizing factor dominates the quadratic terms in the cost function
is described in Section \ref{spar-patt-iden-sec}. Numerical examples
illustrating the proposed methods are provided in
\ref{nume-exam-sec}. Finally, Section \ref{conc-futu-sec} concludes
the paper.

{\bf Notations:}  Throughout the paper, the following notations are
adopted.  The space of $n$ by $m$ matrices with real entries is indicated by
$\mathbb{R}^{n\times m}$. The $n$ by $n$ identity matrix is denoted
$I_{n}$. Operators $\mathbf{Tr(.)}$ and $\mathbf{rank}(.)$ denote the
trace and rank of the matrix operands. The transpose and vectorization operators are denoted by $(.)^{\text T}$ and $\mathbf{vec}(.)$, respectively. The Hadamard product is represented by $\circ$. A matrix is said
to be Hurwitz if all its eigenvalues lie in the open left half of the
complex plane.  $\|.\|_0$ represents the cardinality of a
vector/matrix, while $\|.\|_1$ and $\|.\|_F$ denote $\ell_1$ and
Frobenius norm operators.Also, the norm
$\|.\|_{L^{q}_{\infty}(\mathbb{R}^n)}$ is defined by
\begin{align*}
\|x\|_{L^{q}_{\infty}(\mathbb{R}^n)}\triangleq \sup_{t\geq 0} \|x(t)\|_q
\end{align*}
A real symmetric matrix is said to be positive definite (semi-definite) if all its eigenvalues are positive (non-negative). $\mathbb{S}^n_{++}$ ($\mathbb{S}^n_{+}$) denotes the space of positive definite (positive semi-definite) real symmetric matrices, and the notation $X\succeq Y$ ($X\succ Y$) means $X-Y\in \mathbb{S}^n_+$ ($X-Y\in \mathbb{S}^n_{++}$).
%
%%%%%%%%%%%%%%%%%%%%%%%%%%%%%%%%%%%%%%%%%%%%%%%%%%%%%%%%%%%%%%%%%%%%%%%%%%%%%%%%%%%%%%%%%%%%%%%%%%%%%%%%%%%%%%%%%%%%%%%%%%%%%%%%%%%%%%%%%%
%
%
%
\allowdisplaybreaks
\section{Problem Formulation \label{prob-formu-sec}}
\noindent Let a linear time invariant system be given by its state space realization
\begin{align} \label{eq:system}
\left\{\begin{array}{l}
\dot{x}(t)=Ax(t)+Bu(t)\\
y(t)=Cx(t)
\end{array}\right.
\end{align}
where $x(t)\in \mathbb{R}^n$ is the state vector, $y(t) \in
\mathbb{R}^p$ is the output of the system, $u(t)\in \mathbb{R}^m$ is
the control input, and matrices $A$, $B$ and $C$ have appropriate
dimensions. We consider designing a constant gain output
feedback stabilizing controller
\begin{align}\label{eq:controller}
u(t)=Ky(t), \:\:\: K\in \mathcal{K}
\end{align}
with the minimum number of non-zero entries that minimizes a
quadratic objective function. We further assume that the set of all acceptable \emph{a priori} specified structures for feedback gains, denoted by $\mathcal{K}$, is a convex set. The reason behind this call is that such an assumption not only reduces the complexity of the problem, but also convex constraints on controller constraints have broad real-world
applications. For example, there exist numerous applications in which establishing a link between two particular nodes is impractical either due to physical constraints or extremely high costs; such limitations can be incorporated into the design process by imposing the convex constraints that the corresponding entry of the controller gain should be zero. Also, other regularly occurring limitations such as upper bounds on the entries of the controller matrix can be also be implemented by convex constraints on matrix $K$.

In addition, we consider an upper bound on the norm of the control
input $u(t)$ and the closed loop system output $y(t)$. The search for
such a controller can be formulated as an optimization problem, in which
the sparsity of the feedback gain is incorporated by adding the
$\ell_0$-norm of the gain matrix to the objective function. The
$\ell_0$-norm denotes the cardinality of the feedback gain, hence, it
penalizes the number of non-zero entries of the matrix. Therefore, we
have the following optimization problem
\begin{align}
\min_{K,x,u}\:\: &J=\int_0^\infty [x(t)^{\text T}Qx(t)+u(t)^{\text T}Ru(t)]dt+\lambda\|K\|_0 \tag{{\bf P1}} \label{eq:P1} \\
\mbox{s.t.}\:\:\: &\dot{x}(t)=Ax(t)+Bu(t),\:\:\: x(0)=x_0 \notag \\
&u(t)=KCx(t),\:\:\: K\in\mathcal{K},\notag\\
&\|u\|_{L^{q}_{\infty}(\mathbb{R}^m)}\leq u_{\text max},\:\:\: \|y\|_{L^{q}_{\infty}(\mathbb{R}^p)}\leq y_{\text max},\notag
\end{align}
where $Q \in \mathbb{R}_{+}^{n}$ and $R \in \mathbb{R}_{++}^{m}$ are
performance weight matrices, $x_0$ is the initial state, and
$\lambda\in \mathbb{R}_+$ is the regularization parameter. Also, the value of $q$ in
the norm $\|.\|_{L^{q}_{\infty}(\mathbb{R}^n)}$ can be either infinity
or two. It is possible to rewrite our main optimization problem in the
following equivalent form \cite{Feron:1992}.
% \begin{align*}
%   J&=\int_0^\infty \mathbf{Tr}[x(t)^{\text T}(Q+K^{\text T}RK)x(t)]dt+\lambda\|K\|_0\\
%   &=\mathbf{Tr}[(Q+K^{\text T}RK)\int_0^\infty (x(t)x(t)^{\text T})dt]+\lambda\|K\|_0\\
%   &=\mathbf{Tr}[(Q+K^{\text T}RK)\int_0^\infty
%   (e^{(A+BK)t}x_0)x_0^{\text T}e^{(A+BK)^{\text
%   T}t})dt]+\lambda\|K\|_0
% \end{align*}
% %
% Assuming the asymptotic stability of the closed loop system under
% the state feedback $K$ is guaranteed, there exist a symmetric matrix
% $X_{11}$ satisfying the following equation \cite[p.~11]{Gen_Ric}
% \begin{align}\label{eq:Ricc}
% (A+BK)X_{11}+X_{11}(A+BK)^{\text T}=-x_0x_0^{\text T}
% \end{align}
% %
% Plugging the left hand side of equation (\ref{eq:Ricc}) into our
% cost function, the integrand can be easily integrated as follows
% \begin{align*}
%   J&=\mathbf{Tr}\left[-(Q+K^{\text T}RK)\int_0^\infty (e^{(A+BK)t}\left[(A+BK)X_{11}+X_{11}(A+BK)^{\text T}\right]e^{(A+BK)^{\text T}t})dt\right]\\
%   &\hspace{.15in}+\lambda\|K\|_0\\
%   &=\mathbf{Tr}\left[(Q+K^{\text T}RK)\int_0^\infty \frac{-d}{dt}(e^{(A+BK)t}X_{11}e^{(A+BK)^{\text T}t})dt\right]+\lambda\|K\|_0\\
%   &=\mathbf{Tr}\left[-(Q+K^{\text T}RK)e^{(A+BK)t}X_{11}e^{(A+BK)^{\text T}t}\right]_0^\infty+\lambda\|K\|_0\\
%   &=\mathbf{Tr}\left[(Q+K^{\text T}RK)X_{11}\right]+\lambda\|K\|_0
% \end{align*}
% %
% The last equality holds, since the controller is assumed to stabilize the system, i.e. $e^{(A+BK)t}$ vanishes as $t$ tends to infinity.
% Hence, the following minimization is equivalent to our main optimization problem.
\begin{align}
\min_{X_{11},K}\:\:&\mathbf{Tr}[QX_{11}]+\mathbf{Tr}[RKCX_{11}C^{\text T}K^{\text T}]+\lambda\|K\|_0 \\
\mbox{s.t.}\:\:\:&(A+BKC)X_{11}+X_{11}(A+BKC)^{\text T}+x_0x_0^{\text T}=0, \notag\\
&\mbox{$(A+BKC)$ Hurwitz,}\notag\\
&K\in \mathcal{K},\notag\\
&\|KCx\|_{L^{q}_{\infty}(\mathbb{R}^m)}\leq u_{\text max},\:\:\: \|Cx\|_{L^{q}_{\infty}(\mathbb{R}^p)} \leq y_{\text max}.\notag
\end{align}
The feedback gain matrix $K$ derived from solving the above
optimization problem depends on the value of the initial state
$x_0$. To avoid re-solving the minimization problem for every value of
$x_0$, we design a state feedback controller which minimizes the
expected value of the cost function assuming that the entries of $x_0$
are independent Gaussian random variables with zero mean and
covariance matrix equal to the positive definite matrix $N$, i.e. $x_0\in \mathcal{N}(0,N)$. Using Lyapunov stability
theorem, it can be easily checked that the global asymptotic stability
of the closed loop system is guaranteed if and only if the matrix
$X_{11}$ is positive definite, thus we can rewrite the optimization
problem as follows
\begin{subequations}\label{eq:non_lin_SDP}
\begin{alignat}{1}
\min_{X_{11},X_{12},X_{22},K} &\mathbf{Tr}[QX_{11}]+\mathbf{Tr}[RX_{22}]+\lambda\|K\|_0 \label{eq:non_lin_SDP_a}\\
\mbox{s.t.}\:\:\: &AX_{11}+X_{11}A^{\text T}+BX_{12}^{\text T}+X_{12}B^{\text T}+N=0, \label{eq:non_lin_SDP_b}\\
&X_{11} \succ   0, \label{eq:non_lin_SDP_c}\\
&K\in \mathcal{K}, \label{eq:non_lin_SDP_d}\\
&X_{22}=(KC)X_{11}(KC)^{\text T},\:\:\:X_{12}^{\text T}=KCX_{11}, \label{eq:non_lin_SDP_e}\\
&\|KCx\|_{L^{q}_{\infty}(\mathbb{R}^m)} \leq u_{\text max},\:\:\: \|Cx\|_{L^{q}_{\infty}(\mathbb{R}^p)} \leq y_{\text max}, \label{eq:non_lin_SDP_f}
\end{alignat}
\end{subequations}
where $X_{11} \in \mathbb{R}^{n \times n}$, $X_{12} \in \mathbb{R}^{n
  \times m}$, and $X_{22} \in \mathbb{R}^{m \times m}$. In
optimization problem (\ref{eq:non_lin_SDP}), the constraints
(\ref{eq:non_lin_SDP_b}-\ref{eq:non_lin_SDP_d}) are convex,
nevertheless, the constraints (\ref{eq:non_lin_SDP_e}) are nonlinear
and the control input/output constraints (\ref{eq:non_lin_SDP_f}) are
in time domain, hence, the problem is non-convex.
%
%%%%%%%%%%%%%%%%%%%%%%%%%%%%%%%%%%%%%%%%%%%%%%%%%%%%%%%%%%%%%%%%%%%%%%%%%%%%%%%%%%%%%%%%%%%%%%%%%%%%%%%%%%%%%%%%%%%%%%%%%%%%%%%%%%%%%%%%%%
%
%
%
\section{Rank Constrained Formulation \label{rank-cons-sec}}
In traditional $\text{LQR}$ problems, with no input/output constraints, the
nonlinear constraints can be replaced by a linear matrix inequality to
form an equivalent convex problem. However, the addition of the
sparsity penalizing term to the cost function, the existence of
structural constraints on the feedback gain matrix, and incorporation
of input/output bounds differentiate our problem from the conventional
$\text{LQR}$ problem, making the conventional approach inapplicable. Here, we
propose a controller synthesis approach based on the idea that the
non-convex constraints can be replaced by a rank constraint. Before proceeding, lets state the following lemma.
%-----------------------------------------------------------------------------------------------------------------------------------------
\begin{lemma} \label{lem:non_to_rank}
Let $U \in \mathbb{R}^{n \times n}$, $V \in \mathbb{R}^{n \times m}$, $W \in \mathbb{R}^{m \times m}$, and $Y \in \mathbb{R}^{m \times n}$, with  $U \succ 0$.
Then, $\mathbf{rank}(M)=n$ if and only if $W=YUY^{\text T}$ and $V^{\text T}=YU$, where
\begin{align*}
M=\left[\begin{array}{cc} U & V\\V^{\text T} & W\\I_{n}& Y^{\text T} \end{array}\right]
\end{align*}
\end{lemma}
\begin{proof}
Since $\mathbf{rank}(U)=n$, its inverse exists and the matrix $M$ can be decomposed as
\begin{align*}
M=\left[\begin{array}{cc}
I_n& 0\\
\left[\begin{array}{c}
 V^{\text T} \\
 I_n
\end{array}\right]
 U^{-1} & I_{m+n}
\end{array}\right]
\bar{M}
\left[\begin{array}{cc}
I_n & U^{-1}V\\
0& I_m
\end{array}\right],
\end{align*}
where 
\begin{align*}
\bar{M}=\left[\begin{array}{cc}
U& 0\\
0& 
\left[\begin{array}{c}
 W \\
 Y^{\text T}\end{array}\right]-
 \left[\begin{array}{c}
 V^{\text T} \\
 I_n
\end{array}\right]U^{-1}V
\end{array}\right].
\end{align*}
Since the matrices pre/post-multiplied by the matrix $\bar{M}$ are full rank, the matrix $M$ is rank $n$ if and only if the rank of the matrix $\bar{M}$ is $n$, which is equivalent to
\begin{align*}
\left[\begin{array}{c}
 W \\
 Y^{\text T}\end{array}\right]-
 \left[\begin{array}{c}
 V^{\text T} \\
 I_n
\end{array}\right]U^{-1}V=0_{2n+m}.
\end{align*}
This completes the proof of the lemma.
\end{proof}
The following corollary is now immediate.
\begin{corollary}\label{cor:rank_sq_mat}
Assuming $X_{11}\succ 0$, the constraint
\begin{align*}
\mathbf{rank}\left[\begin{array}{ccc} X_{11}& X_{12}& I_{n}\\X_{12}^{\text T}&X_{22} &(KC)\\I_{n}& (KC)^{\text T} & Z \end{array}\right]=n
\end{align*}
is equivalent to 
\begin{align*}
\left\{\begin{array}{ll}
\mathbf{rank}\left[\begin{array}{cc} X_{11}& X_{12}\\X_{12}^{\text T}&X_{22}\\I_{n}& (KC)^{\text T} \end{array}\right]=n,\\
Z=X_{11}^{-1}
\end{array}
\right.
\end{align*}
\end{corollary}
For legibility purposes, we first develop the equivalent formulation for
the case with no constraint imposed on the control inputs/outputs,
then, we incorporate the bounds on the input/output of the closed loop
system. 
%-----------------------------------------------------------------------------------------------------------------------------------------
\subsection{Rank Constraint Formulation with no Input/Output Constraint \label{rank-cons-sec_a}}
Assuming that no upper bound is defined for the input/output of the controlled system, the next proposition states that the nonlinear Semidefinite Program (\ref{eq:non_lin_SDP}) can be cast as an optimization problem, where all constraints are convex except one, which is a rank constraint.
%-----------------------------------------------------------------------------------------------------------------------------------------
\begin{proposition}
The optimization program (\ref{eq:non_lin_SDP_a}-\ref{eq:non_lin_SDP_e}) is equivalent to the following rank constrained problem
\begin{align}\label{eq:rank_const_SDP}
\min_{X_{11},X_{12},X_{22}, K} &\mathbf{Tr}[QX_{11}]+\mathbf{Tr}[RX_{22}]+\lambda\|K\|_0  \tag{{\bf P2}}\\
\mbox{s.t.}\:\:\: &AX_{11}+X_{11}A^{\text T}+BX_{12}^{\text T}+X_{12}B^{\text T}+N=0, \notag\\
&X_{11} \succ 0,\notag\\
&K \in \mathcal{K},\notag\\
&\mathbf{rank}(X)=n,\notag
\end{align}
where
\begin{align*}
X=\left[\begin{array}{ccc} X_{11}& X_{12}& I_{n}\\X_{12}^{\text T}&X_{22} &(KC)\\I_{n}& (KC)^{\text T} & Z \end{array}\right].
\end{align*}
\end{proposition}
%-----------------------------------------------------------------------------------------------------------------------------------------
\begin{proof}
Applying Lemma \ref{lem:non_to_rank} to the constraints $X_{22}=(KC)X_{11}(KC)^{\text T}$ and $X_{12}^{\text T}=(KC)X_{11}$, they can be equivalently replaced by the rank constraint
\begin{align*}
\mathbf{rank}\left[\begin{array} {cc} X_{11}& X_{12}\\X_{12}^{\text T}&X_{22}\\I_{n}& (KC)^{\text T} \end{array}\right]=n,
\end{align*}
since $X_{11}$ is constrained to be positive definite. Introducing the auxiliary matrix variable $Z$, we can employ Corollary \ref{cor:rank_sq_mat} to rewrite the above rank constraint as a rank constraint on a symmetric matrix, i.e. $\mathbf{rank}(X)=n$.
\end{proof}
It should be noted that augmenting the matrix $\left[\begin{smallmatrix}  I_{n}& KC & Z \end{smallmatrix}\right]^{\text T}$ to the original rank constrained matrix only adds some redundant constraints along with an extra variable. Although we increase the number of variables by introducing the new \emph{n-by-n} variable $Z$, having a symmetric rank constrained matrix has proved to be helpful, as we aim to use a positive semidefinite relaxation of the rank constraint later in this paper, thus, it is crucial to associate the rank constraint to a symmetric matrix.
\begin{remark}\label{rem:inverse_x}
The optimal value of $Z$ in problem (\ref{eq:rank_const_SDP}) is the inverse of the optimal $X_{11}$, i. e. $Z^*={X_{11}^{*}}^{-1}$.
\end{remark}
\subsection{Feasibility of the output feedback control problem}
Before proceeding with the addition of the input/output constraints to the problem formulation, we discuss how our proposed rank constrained reformulation can be exploited in investigating the feasibility of the output feedback control problem under constraints such controller pre-defined structure and input/output constraint. We further utilize the proposed formulation to obtain the upper/lower bounds for the optimal cost of the optimal sparse output feedback control problem. Although the results in this section are stated for the case where no input/output bound is enforced on the controller, they can be effortlessly extended to incorporate such constraints. The next theorem introduces a feasibility test for the existence of a stabilizing output feedback controller with predefined structure.
\begin{theorem}
The linear time invariant system (\ref{eq:system}) can be stabilized using the output feedback controller described in (\ref{eq:controller}) if and only if the optimal cost of the following optimization problem is equal to zero.
\begin{align}\label{eq:feasibility}
\min_{X,Y}\:\:\:&\mathbf{Tr}(Y^{\text T} X)\\
\mbox{s.t.}\:\:\: &AX_{11}+X_{11}A^{\text T}+BX_{12}^{\text T}+X_{12}B^{\text T}\prec 0, \notag\\
&X_{11} \succ 0,\notag\\
&K \in \mathcal{K},\notag\\
&X\succeq 0,\notag\\
&0\preceq Y\preceq I_{2n+m},\notag\\
& \mathbf{Tr}(Y)=n+m,\notag
\end{align}
where
\begin{align*}
X=\left[\begin{array}{ccc} X_{11}& X_{12}& I_{n}\\X_{12}^{\text T}&X_{22} &(KC)\\I_{n}& (KC)^{\text T} & Z \end{array}\right].
\end{align*}
\end{theorem}
\begin{proof}
Applying the results from \cite[p.266]{Dattoro:2005}, if the  matrix $X$ is positive semidefinite, i.e. $X\in\mathbb{S}_+^{2n+m}$, we have
\begin{align*}
\sum_{i=n+1}^{2n+m} \lambda_i(X)=\min_{Y\in \mathbb{R}^{2n+m}}\:\:\:&\mathbf{Tr}(Y^{\text T} X)\\
\mbox{s.t.}\:\:\: &0\preceq Y\preceq I_{2n+m},\notag\\
& \mathbf{Tr}(Y)=(2n+m)-n,\notag
\end{align*}
where $\lambda_1(X)\geq \cdots \geq \lambda_{2n+m}(X)$ are the eigenvalues of $X$. Due to positive semidefiniteness of $X$, the optimal cost of (\ref{eq:feasibility}) is lower bounded by zero. Now, using our rank constraint formulation, it can be verified that such an output feedback controller, satisfying the predefined structure, stabilizes the LTI system (\ref{eq:system}) if and only if the feasible set of (\ref{eq:feasibility}) contains at least a matrix $X$ with rank $n$ for which the sum of  $n+m$ smaller eigenvalues is equal to zero, i.e. $\sum_{i=n+1}^{2n+m} \lambda_i(X)=0$. 
\end{proof}
The optimization problem (\ref{eq:feasibility}) is non-convex due to the existence of the the bi-linear term in its cost function. However, it can be solved utilizing a globally convergent optimization algorithm, which iteratively solves the problem for $X$ and $Y$ till it reaches the convergence \cite{Delgado:2014, Gorski:2007}.  

Next, we investigate the bounds on the optimal cost of the optimization problem \ref{eq:rank_const_SDP}. Assuming feasibility, the lower bound for the optimal cost can be evidently achieved by relaxing the rank constraint $\mathbf{rank}(X)$ by the positive semidefinite constraint $X\succeq0$, since the PSD constraint defines a super-set for the set determined by the rank constraint. As a result, the feasible set of the rank constraint optimization \ref{eq:rank_const_SDP} is a subset of the feasible set of the relaxed problem, hence, the optimal cost of the relaxed optimization provides us with a lower bound for original problem. A more detailed discussion is provided in Section \ref{ADMM-sec}.

As for the upper bound, the results of the theorem \ref{thm:out_up_bound} can be utilized to obtain such a bound if either there is no pre-defined structure on the controller gain or the set of the acceptable controller structures, i.e. $\mathcal{K}$ is assumed to be invariant with respect to positive scaling. This assumption covers the highly applicatory structural constraint, where the feasibility/infeasibility of feedback paths are \emph{a priori} specified generally through a directed graph representation. In such cases the feedback link can be established only if its corresponding edge of the graph, i.e. the pair $(\mathcal{V},\mathcal{E})$ of vertices and edges respectively, is existent, as shown in equation (\ref{eq:pattern_struc}).
\begin{align}\label{eq:pattern_struc}
\mathcal{K}=\{K\:|\: K_{ij}=0\:\:\mbox{if}\:\: (v_i,v_j)\notin\mathcal{E}\}
\end{align}

\begin{theorem}\label{thm:out_up_bound}
Assuming the set $\mathcal{K}$ is invariant under positive scaling, the optimal cost of the following optimization problem provides an upper bound for the solution of the rank constrained problem (\ref{eq:rank_const_SDP}).
\begin{align}  \label{eq:out_up_bound}
\min_{{X}}\:\: &\mathbf{Tr}[RX_{22}]+\mathbf{Tr}[QX_{11}]+\lambda\|\tilde{K}\|_0  \\
\mbox{s.t.}\:\:\: &AX_{11}+X_{11}A^{\text T}+BX_{12}^{\text T}+X_{12}B^{\text T}+N=0, \notag\\
&X_{11} \succ 0,\:\:\: \tilde{K}\in \mathcal{K},\:\:\: \alpha>0,\notag\\
&{X}\succeq 0,\notag
\end{align}
where
\begin{align*}
{X}=\left[\begin{array}{ccc} X_{11}& X_{12}& \alpha I_n\\X_{12}^{\text T}&X_{22} &(\tilde{K}C)\\ \alpha I_n& (\tilde{K}C)^{\text T} & 2\alpha I_n - X_{11} \end{array}\right].
\end{align*}
\end{theorem}
\begin{proof}
The proof is similar to that of the theorem \ref{thm:stat_up_bound}; hence, omitted.
\end{proof}
Next, we state another version of the previous theorem which is valid when the state feedback controller design is intended. Due to the simpler structure of this problem, a tighter upper bound can be achieved by employing the following theorem.
\begin{theorem}\label{thm:stat_up_bound}
Assuming the set $\mathcal{K}$ is invariant under positive scaling, the optimal cost of the following optimization problem provides an upper bound for the solution of the rank constrained problem (\ref{eq:rank_const_SDP}) in the case of feedback controller synthesis, i.e. $C=I_n$.
\begin{align}  \label{eq:stat_up_bound}
\min_{{X}}\:\: &\mathbf{Tr}[RX_{22}]+\mathbf{Tr}[QX_{11}]+\lambda\|\tilde{K}\|_0  \\
\mbox{s.t.}\:\:\: &AX_{11}+X_{11}A^{\text T}+BX_{12}^{\text T}+X_{12}B^{\text T}+N=0, \notag\\
&X_{11} \succ 0,\:\:\: \tilde{K}\in \mathcal{K},\notag\\
&\Gamma=\mathbf{diag}(\alpha_1,\cdots,\alpha_n)\succ 0,\notag\\
&{X}\succeq 0,\notag
\end{align}
where
\begin{align*}
{X}=\left[\begin{array}{ccc} X_{11}& X_{12}& \Gamma\\X_{12}^{\text T}&X_{22} &\tilde{K}\\ \Gamma& \tilde{K}^{\text T} & 2\Gamma - X_{11} \end{array}\right].
\end{align*}
\end{theorem}
\begin{proof}
First, we show that the following optimization problem solves the optimal state feedback control problem, i.e. (\ref{eq:rank_const_SDP}) with $C=I_n$.
\begin{align}  \label{eq:inverse_const_SDP}
\min_{{X}}\:\: &\mathbf{Tr}[RX_{22}]+\mathbf{Tr}[QX_{11}]+\lambda\|\tilde{K}\|_0  \\
\mbox{s.t.}\:\:\: &AX_{11}+X_{11}A^{\text T}+BX_{12}^{\text T}+X_{12}B^{\text T}+N=0, \notag\\
&X_{11} \succ 0,\:\:\: \tilde{K}\in \mathcal{K},\notag\\
&\Gamma=\mathbf{diag}(\alpha_1,\cdots,\alpha_n)\succ 0,\notag\\
&{X}\succeq 0,\notag
\end{align}
where
\begin{align*}
{X}=\left[\begin{array}{ccc} X_{11}& X_{12}&  \Gamma\\X_{12}^{\text T}&X_{22} &\tilde{K}\\ \Gamma& \tilde{K}^{\text T} & \Gamma X_{11}^{-1} \Gamma \end{array}\right].
\end{align*}
Scaling the last block-row and column of the matrix $X$ in (\ref{eq:rank_const_SDP}), assuming the scaler is not zero, does not affect the rank constraint, thus we can equivalently rewrite it as
\begin{align*}
\mathbf{rank}\left[\begin{array}{ccc} X_{11}& X_{12}& \Gamma \\X_{12}^{\text T}&X_{22} &K \Gamma\\ \Gamma& \Gamma {K}^{\text T} & \Gamma X_{11}^{-1}\Gamma \end{array}\right]=n
\end{align*}
The $\ell_0$-norm is invariant under positive scaling, hence $\|{K\Gamma}\|_0=\|K\|_0$. Also, Since the set $\mathcal{K}$ is assumed to be invariant under positive scaling, the constraint ${K}\in \mathcal{K}$ is identical to the matrix $ K\Gamma$ belonging to the set of admissible controller structures. Therefore, it is possible to rewrite the optimization problem in terms of the new variable, defined as $\tilde{K}= K \Gamma$.

In optimization problem (\ref{eq:inverse_const_SDP}), due to the positive definiteness of $X_{11}$, the constraint ${X}\succeq 0$ is equivalent to positive definiteness of its Schur complement, that is 
\begin{align*}
&\left[\begin{array}{cc} X_{22} &\tilde{K}\\ \tilde{K}^{\text T} & \Gamma X_{11}^{-1}\Gamma \end{array}\right]-\left[\begin{array}{cc}
 X_{12}^{\text T}  \\
 \Gamma
\end{array}\right]X_{11}^{-1}\left[\begin{array}{cc}
X_{12} & \Gamma 
\end{array}\right]\succeq 0\\
\Rightarrow 
&\left[\begin{array}{cc} X_{22}-X_{12}^{\text T}X_{11}^{-1}X_{12} &\tilde{K}-X_{12}^{\text T}X_{11}^{-1}\Gamma\\ \tilde{K}^{\text T}-\Gamma X_{11}^{-1}X_{12} & 0 \end{array}\right]\succeq 0
\end{align*}
which holds if and only if $\tilde{K}= X_{12}^{\text T}X_{11}^{-1}\Gamma$, i.e. ${K}=X_{12}^{\text T}X_{11}^{-1}$, and $X_{22}=X_{12}^{\text T}X_{11}^{-1}X_{12}+M$, where $M\succeq0$. Therefore, the feasible set of (\ref{eq:rank_const_SDP}) is a subset of the feasible set of (\ref{eq:inverse_const_SDP}). To conclude our proof, it suffices to show that the optimal value of $M$ in the optimization problem (\ref{eq:inverse_const_SDP}) is zero.

Lets assume $X^*$ is the optimal solution to (\ref{eq:inverse_const_SDP}), where ${M}^*$ is not zero. The optimal cost corresponding to this optimum, namely $J^*$, becomes
\begin{align*}
J^*&=\mathbf{Tr}[R({X_{12}^*}^{\text T}{X_{11}^*}^{-1}{X_{12}^*}+M^*)]+\mathbf{Tr}[Q{X_{11}^*}]+\lambda\|{K^*}\|_0\\
&=\mathbf{Tr}[R({X_{12}^*}^{\text T}{X_{11}^*}^{-1}{X_{12}^*})]+\mathbf{Tr}[Q{X_{11}^*}]+\lambda\|K^*\|_0+\mathbf{Tr}[Q{M^*}]
\end{align*}
Since $\mathbf{Tr}[Q{M^*}]\geq 0$, setting ${M^*}=0$, along with the same values of ${X_{11}^*}$, ${X_{12}^*}$, and ${K^*}$, also belonging to the feasible set, generates a lower cost. This contradicts the optimality of $X^*$, hence the optimal value of $M$ must be zero. Therefore, the optimization problem (\ref{eq:inverse_const_SDP}) solves the problem (\ref{eq:rank_const_SDP}) when $C=I_n$.

For the positive definite matrix $X_{11}$ and the positive scaler $\Gamma$, we have the matrix identity $\Gamma^{-\frac{1}{2}}X_{11}\Gamma^{-\frac{1}{2}}+\Gamma^{\frac{1}{2}}X_{11}^{-1}\Gamma^{\frac{1}{2}}\succeq 2I$. Thus, we can write
\begin{align*}
\Gamma X_{11}^{-1} \Gamma \succeq 2 \Gamma - X_{11}
\end{align*}
Therefore, the feasible set of (\ref{eq:stat_up_bound}) is a subset of the feasible set of the optimization problem (\ref{eq:inverse_const_SDP}). The rest of the proof is straightforward.
\end{proof}
%
%
%**************************************************************************
\subsection{Rank Constraint Formulation in Presence of Input/Output Constraints \label{rank-cons-sec_b}}
Next, we present how an upper bound on the norm of the control input/output can be incorporated into our rank constrained formulation. It is known that for the positive scaler $\gamma$ satisfying $x_0^{\text T}X_{11}^{-1}x_0\leq \gamma^{-1}$, where $x_0$ is the initial state of the system and $X_{11}$ is the solution to the Lyapunov stability condition, the set 
\begin{align}
\mathcal{M}=\{x\in \mathbb{R}^n\:|\: x^{\text T}X_{11}^{-1}x\leq \gamma^{-1}\}
\end{align}
is an invariant set for the closed loop system. Employing the concept of invariant sets for linear systems, we can develop the rank constraint formulation of control system with bounded input norms. The details for two choices of norms utilize to bound the control input in given in the sequel. 
\begin{itemize}
\item {\bf System Norm:} The next theorem describes how the upper bound on the system norm of the control input can be incorporated into the controller synthesis problem using our proposed rank constrained formulation.
\begin{theorem}
The optimization problem (\ref{eq:rank_const_SDP}) can be modified to conservatively incorporate an upper bound on the system norm of the control input, i.e. $\|u\|_{L^{2}_{\infty}(\mathbb{R}^m)}\leq u_{\text max}$, as follows.
\begin{align}  
\min_{X}\:\: &\mathbf{Tr}[RX_{22}]+\mathbf{Tr}[QX_{11}]+\lambda\|K\|_0 \label{eq:rank_const_SDP_input_2} \tag{{\bf P3}}\\
\mbox{s.t.}\:\:\: &AX_{11}+X_{11}A^{\text T}+BX_{12}^{\text T}+X_{12}B^{\text T}+N=0, \notag\\
&X_{11} \succ 0,\notag\\
&K\in \mathcal{K},\notag\\
&\left[\begin{array}{cc}W & (KC)^{\text T}\\ (KC) & u_{\text max}^2I_m\end{array}\right]\succeq 0,\notag\\
&x_0^{\text T}Wx_0 \leq 1,\notag\\
&\mathbf{rank}(X)=n,\notag
\end{align}
where $x_0$ denotes the initial state and
\begin{align*}
X=\left[\begin{array}{ccc} X_{11}& X_{12}& I_{n}\\X_{12}^{\text T}&X_{22} &(KC)\\ I_{n}& (KC)^T & Z \\ \gamma I_n & Y & W
\end{array}\right].
\end{align*}
\end{theorem}
\begin{proof}
Based on the lines in \cite[p. 103]{Boyd:1994}, we have
\begin{align*}
\|u\|_{L^{2}_{\infty}(\mathbb{R}^m)}&=\sup_{t\geq 0}\:\|u(t)\|_2=\sup_{t\geq 0}\:\|KCx(t)\|_2\\
&\leq \sup_{x \in \mathcal{M}}\:\|KCx\|_2\\
&= \sup_{x \in \mathcal{M}}\:\|KCX_{11}^{1/2}X_{11}^{-1/2}x\|_2\\
&=\sqrt{\lambda_{\text max}(X_{11}^{1/2}(KC)^{\text T}(KC)X_{11}^{1/2})\gamma^{-1}}
\end{align*}
Thus, the input constraint $\|u\|_{L^{2}_{\infty}(\mathbb{R}^m)}\leq u_{\text max}$ holds for all $t\geq 0$ if
\begin{align}
\left[\begin{array}{cc}\gamma X_{11}^{-1} & (KC)^{\text T}\\ (KC)& u_{\text max}^2I_m\end{array}\right]\succeq 0,\notag\\
x_0^{\text T}\gamma X_{11}^{-1}x_0 \leq 1.\notag
\end{align}
The existence of the term $\gamma X_{11}^{-1}$ in the above matrix inequality makes it nonlinear, however, Utilizing Lemma \ref{lem:non_to_rank}, it can be verified that the rank constraint $\mathbf{rank}(X)=n$, applied on the modified matrix $X$, is equivalent to introducing the variables $W=\gamma X_{11}^{-1}$. The rest of the proof is straightforward.
\end{proof}
%***************************************************************************
\item{\bf Infinity Norm:} If the constraint on the control input is in the form of $\|u(t)\|_{L^{\infty}_{\infty}(\mathbb{R}^m)} \leq u_{\text max}$, it can be represented using the following matrix inequalities \cite[p. 104]{Boyd:1994}. 
\begin{align}
\left[\begin{array}{cc}V & KC\\(KC)^{\text T} & \gamma X_{11}^{-1}\end{array}\right]\succeq 0,\notag\\
V_{ii}\leq  u_{\text max}^2\notag\\
x_0^{\text T}\gamma X_{11}^{-1}x_0 \leq 1.\notag
\end{align}
Therefore, this problem can also be posed as a rank constrained problem through the next theorem.
\begin{theorem}
The optimization problem (\ref{eq:rank_const_SDP}) can be modified to conservatively incorporate an upper bound on the infinity norm of the control input, i.e. $\|u\|_{L^{\infty}_{\infty}(\mathbb{R}^m)} \leq u_{\text max}$, as follows.
\begin{align}  
\min_{X}\:\: &\mathbf{Tr}[RX_{22}]+\mathbf{Tr}[QX_{11}]+\lambda\|K\|_0 \label{eq:rank_const_SDP_input_inf} \tag{{\bf P3'}}\\
\mbox{s.t.}\:\:\: &AX_{11}+X_{11}A^{\text T}+BX_{12}^{\text T}+X_{12}B^{\text T}+N=0, \notag\\
&X_{11} \succ 0,\notag\\
&K\in \mathcal{K},\notag\\
&\left[\begin{array}{cc}V & KC\\(KC)^{\text T} & W\end{array}\right]\succeq 0,\notag\\
&V_{ii}\leq  u_{\text max}^2,\notag\\
&x_0^{\text T}Wx_0 \leq 1,\notag\\
&\mathbf{rank}(X)=n,\notag
\end{align}
where $x_0$ denotes the initial state and
\begin{align*}
X=\left[\begin{array}{ccc} X_{11}& X_{12}& I_{n}\\X_{12}^{\text T}&X_{22} &(KC)\\ I_{n}& (KC)^T & Z \\ \gamma I_n & Y & W
\end{array}\right].
\end{align*}
\end{theorem}
\end{itemize}
\begin{remark}
Other norms such as element-wise bound on the control input or the norm bounds on the system outputs can also be assimilated into the rank constraint using similar techniques. The details are omitted with the purpose of improving the readability of the manuscript.
\end{remark}
%***************************************************************************
All of the optimization problems posed so far are NP-hard due to the existence of the $\ell_0$-norm in the cost function and the rank constraint. Therefore, no polynomial time algorithm capable of solving it in its general form, exists. In the next two sections, we propose a method to sub-optimally solve the problem, then, discuss a special case of the problem where only the sparsity of the controller is of importance.
%%%%%%%%%%%%%%%%%%%%%%%%%%%%%%%%%%%%%%%%%%%%%%%%%%%%%%%%%%%%%%%%%%%%%%%%%%%%%%%%%%%%%%%%%%%%%%%%%%%%%%%%%%%%%%%%%%%%%%%%%%%%%%%%%%%%%%%%%%%

\section{Convex Relaxtions of the Optimal Control Problem} \label{sec:conv_relaxation}
In this section, we study the general problem of designing a sparse optimal feedback controller. Although the results we present in the sequel are applicable to the optimization problem (\ref{eq:non_lin_SDP}) in its general form, to enhance the legibility of the paper, we choose to state them in the absence of the constraints on the control inputs and system outputs. Hence, we consider the problem (\ref{eq:rank_const_SDP}), which is a combinatorial problem, due to the existence of the $\ell_0$-norm, in fact a quasi-norm, in the cost and the rank constraint. The weighted $\ell_1$-norm minimization problem is a well-known heuristic for cardinality minimization \cite{Donoho:2006,Candes:2006,Candes:2008}. Although weighted $\ell_1$-norm relaxation does not guarantee the exact optimal controller recovery, it reduces the complexity of the problem substantially. Substituting the cardinality penalizing term with the weighted $\ell_1$-norm of the controller gain matrix, we obtain the following relaxed optimization problem
\begin{align}\label{eq:rank_const_card_SDP_ell_1} 
\min_{X} &\mathbf{Tr}[QX_{11}]+\mathbf{Tr}[RX_{22}]+\lambda\|W\circ K\|_1 \tag{{\bf C1}} \\
\mbox{s.t.}\:\:\: &AX_{11}+X_{11}A^{\text T}+BX_{12}^{\text T}+X_{12}B^{\text T}+N=0, \notag\\
&X_{11}\succ 0,\notag\\
&K\in \mathcal{K},\notag\\
&\mathbf{rank}(X)=n. \notag
\end{align}
where the weight matrix $K$ is a positive matrix with appropriate dimensions and
\begin{align*}
X=\left[\begin{array}{ccc} X_{11}& X_{12}& I_{n}\\X_{12}^{\text T}&X_{22} &(KC)\\I_{n}& (KC)^{\text T} & Z \end{array}\right].
\end{align*}
%

%%%%%%%%%%%%%%%%%%%%%%%%%%%%%%%%%%%%%%%%%%%%%%%%%%%%%%%%%%%%%%%%%%%%%%%%%%%%%%%%%%%%%%%%%%%%%%%%%%%%%%%%%%%%%%%%%%%%%%%%%%%%%%%%%%%%%%%%%%%
Combinatorial nature of the our rank constrained problem, makes the search for the optimal point computationally intractable. Therefore, a systematic solution to general rank constrained problem has remained open \cite{Yu:2011,Goemans:1995}. Nonetheless, attempts have been made to solve specific rank constrained problems, and algorithms proposed to locally solve such problems \cite{Delgado:2014, Orsi:2006}. Here, we propose to use a particular form of Alternating Direction Method of Multipliers (ADMM) to solve our rank constrained problem.

\subsection{ADMM for Solving the Relaxed Problem \label{ADMM-sec}}
ADMM was originally developed in 1970s \cite{Glowinski:1975,Gabay:1976}, and has been used for optimization purposes since. Boyd \emph{et al.}, in \cite{Boyd:2011}, argued that this method can be efficiently applied to large-scale optimization problems. For non-convex problems, the convergence of ADMM is not guaranteed, also, it may not reach the global optimum when it converges, thus, the convergence point should be considered as a local optimum. 

For the optimization problem (\ref{eq:rank_const_card_SDP_ell_1}), one way to perform convex relaxation is replacing the rank constraint on matrix $X$ with a positive semi-definite constraint, i.e. $X\succeq 0$. Since $X_{11}$ is positive definite, using lemma \ref{lem:non_to_rank}, it can be seen that the rank constraint in (\ref{eq:rank_const_card_SDP_ell_1}) is equivalent to 
\begin{align}
\left[\begin{array}{cc} X_{22} &(KC)\\ (KC)^{\text T} & Z \end{array}\right]-\left[\begin{array}{cc}
 X_{12}^{\text T}  \\
 I_n
\end{array}\right]X_{11}^{-1}\left[\begin{array}{cc}
X_{12} & I_n 
\end{array}\right]=0,
\end{align}
which implies that the Schur complement of the matrix $X$ should be equal to zero, while $X\succeq0$ is the same as positive semi-definiteness of its Schur complement. Therefore, the set defined by the PSD constraint is a super-set for the one defined by the rank constraint. Now, if we define the convex set
\begin{align*}
\mathcal{C}=\{X \:| &AX_{11}+X_{11}A^{\text T}+BX_{12}^{\text T}+X_{12}B^{\text T}+N=0, \\
&X_{11} \succ 0,\:\:\:K\in \mathcal{K},\:\:\:X\succeq 0 \}
\end{align*}
and $\mathcal{S}$ denotes the set of $(2n+m)\times (2n+m)$ symmetric matrices with rank equal to $n$, the minimization (\ref{eq:rank_const_card_SDP_ell_1}) can be represented as
\begin{align}
\min_{X}\:\:\:\: &f(X)\\
s.t \:\:\:\: &X \in \mathcal{C} \cap \mathcal{S} \notag
\end{align}
where 
\begin{align*}
f(X)=\mathbf{Tr}[RX_{22}]+\mathbf{Tr}[QX_{11}]+\lambda\|W \circ K\|_1
\end{align*}
and the weight matrix $W$ is a positive real matrix with appropriate dimensions. Considering the above formulation, the ADMM algorithm can be carried out by repeatedly performing the steps stated in the sequel till certain convergence criteria is satisfied \cite[p. 74]{Boyd:2011}.
\begin{subequations}\label{eq:alg}
\begin{alignat}{1}
&X^{(k+1)}=\arg \min_{X\in \mathcal{C}} \:\:\:\: f(X)+(\rho/2)\|X-V^{(k)}+Y^{(k)}\|^2_F \label{ADMM-gen-convex}\\
&V^{(k+1)}=\Pi_\mathcal{S}(X^{(k+1)}+Y^{(k)})\label{ADMM-gen-proj}\\
&Y^{(k+1)}=Y^{(k)}+X^{(k+1)}-V^{(k+1)}\label{ADMM-gen-updat}\\
&w_{ij}^{(k+1)}=\frac{1}{|k_{ij}^{(k)}|+\delta}\label{eq:w_update}
\end{alignat}
\end{subequations}
where $w_{ij}$ and $k_{ij}$ denote the $(i,j)$ entries of the matrices $W$ and $K$, respectively. The convexity of the cost function and the constraints makes (\ref{ADMM-gen-convex}) a convex problem, hence, it can be solved by various computationally efficient methods. The operator $\Pi_\mathcal{S}(.)$, in (\ref{ADMM-gen-proj}), denotes projection onto the set $\mathcal{S}$. Although the projection on a non-convex set is generally not an easy task, it can be carried out exactly in the case of projecting on the set of matrices with pre-defined rank. In our case, the set $\mathcal{S}$ is the set of matrices with rank $n$, thus, $\Pi_\mathcal{S}(.)$ can be determined by carrying out Singular Value Decomposition (SVD) and keeping the top dyads, i.e.
\begin{align}\label{eq:SVD_thresh}
\Pi_\mathcal{S}(X)\triangleq \sum_{i=1}^n \sigma_iu_iv_i^{\text T}
\end{align}
where $\sigma_i$, $i=1,\cdots,n$ are the $n$ largest singular values of matrix $x$, and the vectors $u_i\in\mathbb{R}^{(2n+m)}$ and $v_i \in \mathbb{R}^{(2n+m)}$ are their corresponding left and right singular vectors. 
The step (\ref{ADMM-gen-updat}) in the algorithm is a simple matrix manipulation to update the auxiliary variable $u$, which is exploited in the next iteration.

The last step of the heuristic (\ref{eq:alg}) is to update the weight on the entries of the controller matrix approximately inversely proportional to the value of the corresponding matrix entry recovered from the previous iteration. Hence, the next iteration optimization will be forced to concentrate on the entries with smaller magnitudes, which results in promoting the controller sparsity. It should also be noted the relatively small constant $\delta$ is added to the denominator of the update law (\ref{eq:w_update}) to avoid instability of the algorithm, especially when a recovered controller entry turns out to be zero in the previous iteration \cite{Candes:2008}.

Initializing with the stabilizing $\text{LQR}$ controller along with its corresponding Lyapunov matrix, a sub-optimal minimizer to the problem (\ref{eq:rank_const_card_SDP_ell_1}) can be obtained by iterating the steps (\ref{ADMM-gen-convex}-\ref{ADMM-gen-updat}) until the convergence is achieved. The algorithm's stopping criteria is either reaching the maximum number of iterations or $\varepsilon<\varepsilon^*$, where $\varepsilon$ update is performed using the following equation.
\begin{align}
\varepsilon^{(k+1)} \triangleq \mathbf{max}(\|X^{(k+1)}-V^{(k+1)}\|_F,\|V^{(k+1)}-V^{(k)}\|_F)
\label{eq:stop_cond}
\end{align}
The small enough entries of the generated controller gain can then be truncated to yield a sparse controller matrix, namely $\bar{K}$, while considering the extent of its adverse effect on the stability and performance of the closed loop system. The step-by-step procedure is described in Algorithm 1.
%%%%%%%%%%%%%%%%%%%%%%%%%%%%%%%%%%%%%%%%%%%%%%%%%%%%%%%%%%%%%%%%%%%%%%%%%
\begin{figure}[t]
\begin{center}
\begin{tabular}{| l |}
\hline
{\bf Algorithm 1:} Solution to {\bf \ref{eq:rank_const_card_SDP_ell_1}}\\ \hline\hline
{\bf Inputs:} $A$, $B$, $C$, $Q$, $R$, $\lambda$, $\mathcal{K}$, $\rho$, $\delta$ and $\varepsilon^*$\\
~1: {\em Initialization:} \\
~~~~Find $X^{(0)}$ by solving (\ref{ADMM-gen-convex}) for $\lambda=0,\rho=0$ ($\text{LQR}$),\\
~~~~Set $V^{(0)}=X^{(0)}$, $Y^{(0)}=0\times I_{(2n+m)}$, and $n=0$,\\
~2: {\bf While} $\varepsilon^n \leq \varepsilon^*$  \bf{do}\\
~3: \hspace{.1in} Update $X^{(n)}$ by solving (\ref{ADMM-gen-convex}),\\
~4: \hspace{.1in} Update $V^{(n)}$ using Eq.~(\ref{ADMM-gen-proj}),\\
~5: \hspace{.1in} Update $Y^{(n)}$ using Eq.~(\ref{ADMM-gen-updat}),\\
~6: \hspace{.1in} Update $W^{(n)}$ using Eq.~(\ref{eq:w_update}),\\
~7: \hspace{.1in} Update $\varepsilon^{(n+1)}$ using Eq.~(\ref{eq:stop_cond}),\\
~8: \hspace{.1in} $n\leftarrow n+1$,\\
~9: {\bf end while}\\
10: Truncate $K$,\\
{\bf Output:} $\bar{K}$ \\
\hline
\end{tabular} 
\vspace{-.27in}
\end{center}
\end{figure}
%%%%%%%%%%%%%%%%%%%%%%%%%%%%%%%%%%%%%%%%%%%%%%%%%%%%%%%%%%%%%%%%%%%%%%%%%
As said before, the truncation step in the algorithm should be performed with the necessary precautions, since not only does it deteriorate the obtained optimal performance but it also may destabilize the closed loop system. The following proposition provides the sufficient condition under which the truncation process does not have cause instability in the closed loop system.
\begin{proposition}
The truncated controller, denoted by $\bar{K}$, stabilizes the system if the truncation threshold $\xi$ is bounded by
\begin{align}
\xi<\frac{\sigma_{min}(N)}{\sum_{ij}\|BE_{ij}CX_{11}+X_{11}(BE_{ij}C)^{\text T}\|_2}
\end{align}
where $\sigma_{min}(N)$ denotes the smallest singular value of the matrix $N$, which is the positive definite matrix satisfying 
\begin{align*}
(A+BKC)X_{11}+X_{11}(A+BKC)^{\text T}+N=0,
\end{align*}
and $E_{ij}\in\mathbb{R}^{m\times p}$ is the matrix whose only nonzero entry, equal to $1$, is its $(i,j)$-entry. 
\end{proposition}
\begin{proof}
Defining the matrix of the truncated entries of the controller as $K_{\xi}=K-\bar{K}$, we will have
\begin{align*}
(A+B(\bar{K}+K_{\xi})C)X_{11}+X_{11}(A+B(\bar{K}+K_{\xi})C)^{\text T}+N&=0,\\
(A+B\bar{K}C)X_{11}+X_{11}(A+B\bar{K}C)^{\text T}+BK_{\xi}CX_{11}+X_{11}(BK_{\xi}C)^{\text T}+N&=0.
\end{align*}
Hence, the truncated controller stabilized the system if 
\begin{align*}
BK_{\xi}CX_{11}+X_{11}(BK_{\xi}C)^{\text T}+N\succ 0,
\end{align*}
which is equivalent to the following inequality, for any nonzero vector $x$ with appropriate dimension,
\begin{align*}
x^{\text T}(BK_{\xi}CX_{11}+X_{11}(BK_{\xi}C)^{\text T}+N)x> 0.
\end{align*}
The previous inequality holds if we have
\begin{align*}
|x^{\text T}(BK_{\xi}CX_{11}+X_{11}(BK_{\xi}C)^{\text T})x|<\sigma_{min}(N)x^{\text T}x .
\end{align*}
Noting that $K_{\xi}=\sum_{(i,j)\in \mathcal{D}} k_{ij}E_{ij}$, where $\mathcal{D}=\left\{(i,j)|~|k_{ij}|<\xi\right\}$, we rewrite the above inequality as
\begin{align*}
|x^{\text T}( \sum_{(i,j)\in \mathcal{D}} k_{ij} [ BE_{ij}CX_{11}+X_{11}(BE_{ij}C)^{\text T}])x|<\sigma_{min}(N)x^{\text T}x .
\end{align*}
which is true if
\begin{align*}
\sum_{(i,j)\in \mathcal{D}} |k_{ij}| \|BE_{ij}CX_{11}+X_{11}(BE_{ij}C)^{\text T} \|_2<\sigma_{min}(N).
\end{align*}
Since $|k_{ij}|<\xi$ for all $(i,j) \in \mathcal{D}$, we can conservatively replace the above inequality with
\begin{align*}
\xi \sum_{\forall (i,j)} \|BE_{ij}CX_{11}+X_{11}(BE_{ij}C)^{\text T} \|_2<\sigma_{min}(N),
\end{align*}
which completes our proof.
\end{proof}
\begin{remark}
For the problem of optimal sparse state feedback control design, i.e $C=I_n$, if there exists no \emph{a priori} defined controller structure or the constraint on the controller matrix is in the form of sparsity pattern, one way to perform the truncation is to solve the minimization problem, assuming that all of the variables have already converged to their optimal values except the controller matrix. Thus, we will have
\begin{align}\label{eq:K_trunc}
\min_{K}\:\:\:\: &\lambda\|K\|+(\rho/2)\|K-(K^{(V^*)}-K^{(Y^*)})\|^2_F \\
\mbox{s.t.}\:\:\: &K\in \mathcal{K}\notag.
\end{align}
where $K^{(V^*)}$ and $K^{(Y^*)}$ are the sub-blocks of the optimal values of $V^*$ and $Y^*$, respectively, which correspond to the controller gain matrix, and $\|.\|$ can be chosen as either $\ell_1$ or $\ell_0$-norm. Moreover, in such problems, the problem (\ref{eq:K_trunc}) has a unique solution that can be obtained analytically as follows \cite{Lin:2013,Boyd:2011}. For example, if the norm used in (\ref{eq:K_trunc}) is $\ell_0$-norm, the optimal values of the elements, not constrained to zero, can be obtained through the following element-wise truncation operator  
\begin{align}\label{eq:truncate_l0}
K^*_{ij}=\left\{\begin{array}{ll}
K^{(V^*)}_{ij}-K^{(Y^*)}_{ij},& |K^{(V^*)}_{ij}-K^{(Y^*)}_{ij}|>\sqrt{2\lambda/\rho}\\
0,&{\text otherwise.}
\end{array}
\right.
\end{align}
\end{remark}

%%%%%%%%%%%%%%%%%%%%%%%%%%%%%%%%%%%%%%%%%%%%%%%%%%%%%%%%%%%%%%%%%%%%%%%%%%%%%%%%%%%%%%%%%%%%%%%%%%%%%%%%%%%%%%%%%%%%%%%%%%%%%%%%%%%%%%%%%%%
%
\section{Sparsest Stabilizing Output Feedback Controller Design \label{spar-patt-iden-sec}}
Next, we study the special case in which obtaining a stabilizing constant gain feedback controller with the sparsest feasible structure, i.e. considering the  constraints, is desirable. To this end, we eliminate the terms which penalize the system performance from the cost, i.e. both $R$ and $Q$ are zero. One of the applications that can be addressed using this problem setup is the problem of stabilizing controller synthesis for networks/systems where establishing communication links between nodes are so costly that the control effort and error cost are almost negligible. Having $R=0$, it can be seen the variable $X_{22}$ is irrelevant in this case, so its corresponding constraints can be removed from the optimization program. Therefore, we will have
 \begin{align}\label{eq:rank_const_SPI}
\min_{X_{11},X_{12},K,N} &\|K\|_0  \tag{{\bf P4}}\\
\mbox{s.t.}\:\:\: &AX_{11}+X_{11}A^{\text T}+BX_{12}^{\text T}+X_{12}B^{\text T}+N=0, \notag\\
&X_{11}\succ 0,~N\succ 0\notag\\
&K \in \mathcal{K}, \notag\\
&\mathbf{rank}\left[\begin{array}{cc} X_{11}& X_{12}\\I_{ n}& (KC)^{\text T} \end{array}\right]=n, \notag
\end{align}
The following lemma helps us convert rank constrained cardinality minimization problem (\ref {eq:rank_const_SPI}) into an affine rank minimization problem.
%-----------------------------------------------------------------------------------------------------------------------------------------
\begin{lemma} \label{lem:rank_to_cost}
Consider the following rank constrained cardinality minimization problem
\begin{align}\label{eq:lem_rank_to_cost_a}
\min_{Y}\:\: &\|W_1YW_2\|_0 \\
\mbox{s.t.}\:\:\: &\mathcal{L}_1(Y)=\mu, \notag\\
&\mathcal{L}_2(Y)\succeq 0, \notag \\
&\mathbf{rank}(Y)=\mathbf {rank}(Y_{11})=n, \notag
\end{align}
where $Y$ is partitioned as $Y=\left[ \begin{smallmatrix} Y_{11}& Y_{12}\\Y_{21} & Y_{22} \end{smallmatrix}\right]\in \mathbb{R}^{p\times q}$, $W_1 \in \mathbb{R}^{a \times p}$ and $W_2 \in \mathbb{R}^{q \times b}$ are weight matrices, $\mathcal{L}_1$ and $\mathcal{L}_2$ are two arbitrary maps, and $Y_{11}\in \mathbb{R}^{n \times n}$ is a full rank square matrix ($n<min\{p,q\}$). If the optimization problem (\ref{eq:lem_rank_to_cost_a}) is feasible, it can be equivalently formulated as
\begin{align} \label{eq:lem_rank_to_cost_b}
\min_{Y}\:\: &\|W_1YW_2\|_0+\nu  \mathbf{rank} (Y) \\
\mbox{s.t.}\:\:\: &\mathcal{L}_1(Y)=\mu, \notag\\
&\mathcal{L}_2(Y)\succeq 0, \notag\\
&\mathbf {rank}(Y_{11})=n, \notag
\end{align}
for any $\nu>ab$.
\end{lemma}
\begin{proof}
Let $Y^*$ be the optimum of (\ref{eq:lem_rank_to_cost_a}), then $\mathbf{rank}(Y^*_{11})=n$ and it satisfies both equality and inequality constraints. Therefore, it belongs to the feasible set of (\ref{eq:lem_rank_to_cost_b}). Furthermore, for every point $Y$ in the feasible set of (\ref{eq:lem_rank_to_cost_b}) with the rank greater than $n$, we have
\begin{align*}
J-J^*&=\|W_1YW_2\|_0+\nu \mathbf{rank}(Y)\\
&\hspace{.15in}-(\|W_1Y^*W_2\|_0+\nu \mathbf{rank}(Y^*))\\
&=(\|W_1YW_2\|_0-\|W_1Y^*W_2\|_0)\\
&\hspace{.15in}+\nu(\mathbf{rank}(Y)-\mathbf{rank}(Y^*))\\
&\geq -\|W_1Y^*W_2\|_0+\nu(\mathbf{rank}(Y)-\mathbf{rank}(Y^*))
\end{align*}
Since $W_1Y^*W_2\in \mathbb{R}^{a\times b}$, it is safe to bound the cardinality as $\|W_1Y^*W_2\|_0\leq ab$. Using  $\mathbf{rank}(Y)-\mathbf{rank}(Y^*)\geq 1$, we can write
\begin{align*}
J-J^*&> -ab+\nu
\end{align*}
Hence, the cost for all $Y$, with rank greater than $n$, is higher than the cost of $Y^*$, if $\nu>ab$. This means the optimum of (\ref{eq:lem_rank_to_cost_b}) should be of rank $n$. Knowing that $Y^*$ has the minimum cardinality among the matrices with rank equal to $n$, we conclude that $Y^*$ is also the optimum for (\ref{eq:lem_rank_to_cost_b}).

Conversely, let ${\bar{Y}}$ be the optimal point for (\ref{eq:lem_rank_to_cost_b}). As it is shown in the first part of the proof, the cost generated by matrices, with the rank higher than $n$ is greater than that of rank $n$ matrices, for $\nu>ab$. Thus, the rank of $\bar{Y}$ must be $n$, unless no point with the rank equal to $n$ exists in the feasible set of (\ref{eq:lem_rank_to_cost_b}). However, this implies that (\ref{eq:lem_rank_to_cost_a}) is infeasible, which contradicts the lemma's assumption. Therefore, $\bar{Y}$ is the minimizer of the cardinality term of the cost function among all rank $n$ matrices in the feasible set of (\ref{eq:lem_rank_to_cost_b}), i.e. $\bar{Y}$ is the minimizer of (\ref{eq:lem_rank_to_cost_a}).
\end{proof}
%-----------------------------------------------------------------------------------------------------------------------------------------
%
\begin{remark}
In the optimization problem (\ref{eq:lem_rank_to_cost_a}), if the cost which is to be minimized is the rank of the matrix $W_1YW_2$, instead of its cardinality, lemma \ref{lem:rank_to_cost} can still be applied to the problem for any $\nu>min\{a,b\}$.
\end{remark}

Applying lemma \ref{lem:rank_to_cost} to (\ref{eq:rank_const_SPI}), we can equivalently write it as
\begin{align}\label{eq:l0_rank_cost_SPI}
\min_{X_{11},X_{12}, K,N} &\|K\|_0+ \nu \mathbf{rank}\left[\begin{array}{cc} X_{11}& X_{12}\\I_{ n}& (KC)^{\text T} \end{array}\right] \\
\mbox{s.t.}\:\:\: &AX_{11}+X_{11}A^{\text T}+BX_{12}^{\text T}+X_{12}B^{\text T}+N=0, \notag\\
& \mathbf{diag}(X_{11},N) \succ 0, \notag\\
& K\in \mathcal{K},\notag
\end{align}
with $\nu >mn$. Note that the matrix $X_{11}$ is full rank due to its positive definiteness, therefore, all of the requirements of lemma \ref{lem:rank_to_cost} are satisfied. 

%-----------------------------------------------------------------------------------------------------------------------------------------
%
\begin{remark}
The solution to equation (\ref{eq:l0_rank_cost_SPI}) falls into the category of the problem of recovery of simultaneously structured models where the matrix of interest is both sparse and low-rank \cite{Oymak:2012,Chen:2014}. Oymak \emph{et al.}, in their recent paper, have shown that minimizing a combination of the known norm penalties corresponding to each structure (for example, $\ell_1$-norm for sparsity and nuclear norm for matrix rank) will not yield better results than an optimization exploiting only one of the structures. They have concluded that an entirely new convex relaxation is required in order to fully utilize both structures \cite{Oymak:2012}.
\end{remark}
Without loss of generality, the following theorem is stated assuming $m<n$.
%-----------------------------------------------------------------------------------------------------------------------------------------
%
\begin{theorem}
The optimization problem (\ref{eq:rank_const_SPI}), if feasible, is equivalent to 
\begin{align}\label{eq:rank_cost_SPI}
\min_{X_{11},X_{12},C, K,N,\varepsilon} &\mathbf{rank}(\mathbf{diag}[\mathbf{vec}(K),\Psi_1,\cdots,\Psi_\nu,\Phi_1,\cdots,\Phi_\rho])\\
\mbox{s.t.}\:\:\: &AX_{11}+X_{11}A^{\text T}+BX_{12}^{\text T}+X_{12}B^{\text T}+N=0, \notag\\
&K \in \mathcal{K}, \notag\\
&\varepsilon>0, \notag
\end{align}
where 
\begin{align*}
\begin{array}{ll}
\Psi_i=\left[\begin{array}{c|c}\begin{array}{cc} X_{11}& X_{12}\\I_{n}& (KC)^{\text T} \end{array}&0_{(2n \times (n-m))}\end{array}\right]& i=1,\cdots,\nu\\
\Phi_i=\left[\begin{array}{cc}I_{2n} &D\\ D^{\text T} & \mathbf{diag}(X_{11},N)-\varepsilon I_{2n}\end{array}\right]&i=1,\cdots,\rho
\end{array}
\end{align*}
and the parameters $\nu$ and $\rho$ are integers satisfying
\begin{align*}
\rho&>mn+\nu. \mathbf{max}\{2n,(n+m)\}\\
\nu &>mn
\end{align*}
\end{theorem}
\begin{proof}
For a function that maps matrices into $q \times q$ symmetric matrices, positive semi-definiteness can be equivalently expressed as a rank constraint \cite{Recht:2007}
\begin{align}
f(X)\succeq 0 \Longleftrightarrow \mathbf{rank}\left[\begin{array}{cc}I_q &U\\ U^{\text T} & f(X)\end{array}\right]\leq q
\end{align}
for some $U\in \mathbb{R}^{q}$. Since $ \mathbf{diag}(X_{11},N) \succ0$ is equivalent to $ \mathbf{diag}(X_{11},N)  \succeq \varepsilon I_{2n}$ for some $\epsilon>0$, it can be written as the following rank constraint
\begin{align*}
\mathbf{rank}\left[\begin{array}{cc}I_{2n} &D\\ D^{\text T} &  \mathbf{diag}(X_{11},N) -\varepsilon I_{2n}\end{array}\right]=2n
\end{align*}
Noting that the cost function in (\ref{eq:l0_rank_cost_SPI}) is bounded by $mn+\nu.\mathbf{max}\{2n,(n+m)\}$, we can use an argument similar to the one used in the proof of lemma \ref{lem:rank_to_cost} to to show that (\ref{eq:rank_const_SPI}), if feasible, can be equivalently cast in the following form
\begin{align}\label{eq:l0_rank_rank_cost_SPI}
\min_{X_{11},X_{12},C,K,N}&\|K\|_0+ \nu \mathbf{rank}\left[\begin{array}{cc} X_{11}& X_{12}\\I_{n}& (KC)^{\text T} \end{array}\right]\notag\\
&\hspace{0.34in}+\rho  \mathbf{rank}\left[\begin{array}{cc}I_{2n} &D\\ D^{\text T} & M-\varepsilon I_{2n}\end{array}\right]\\
\mbox{s.t.}\:\:\: &AX_{11}+X_{11}A^{\text T}+BX_{12}^{\text T}+X_{12}B^{\text T}+N=0, \notag\\
&K \in \mathcal{K},\notag\\
&\varepsilon>0,\notag
\end{align}
where 
\begin{align*}
\nu&>mn\\
\rho&>mn+\nu. \mathbf{max}\{2n,(n+m)\}.
\end{align*}
Next, we are going to show that the cost function of (\ref{eq:l0_rank_rank_cost_SPI}) is equal to the cost function of (\ref{eq:rank_cost_SPI}) for $\rho$ and $\nu$ chosen to be integers satisfying the conditions. It can be easily verified that $\|K\|_0=\mathbf{rank}(\mathbf{diag}(\mathbf{vec}(K)))$, also, the ranks of the square matrices $\Psi_i$'s are equal to the rank of  $\left[\begin{smallmatrix} X_{11}& X_{12}\\I_{n}& (KC)^{\text T} \end{smallmatrix}\right]$.

If the parameters $\rho$ and $\nu$ are integers, we can construct a block diagonal matrix in the following form 
\begin{align*}
\mathbf{diag}[\mathbf{vec}(K),\Psi_1,\cdots,\Psi_\nu,\Phi_1,\cdots,\Phi_\rho]
\end{align*}
Thus, the rank of such matrix is equal to the sum of the rank of its constructing block matrices. Therefore, it is equal to the cost function of the optimization problem (\ref{eq:rank_cost_SPI}), which completes our proof.
\end{proof}
%-----------------------------------------------------------------------------------------------------------------------------------------
%
The above formulation is in the form of \emph{Affine Rank Minimization Problem} (ARMP), which consists of minimizing the rank of a matrix subject to affine/convex constraints with the general form
\begin{align*}
\min_{X}\:\:&\mathbf{rank}(X)\\
\mbox{s.t.}\:\:\: &\mathcal{A}(X)=b
\end{align*}
for a fixed infinitesimal $\varepsilon>0$. ARMP has been investigated thoroughly in the past decade and several heuristics have been proposed to solve it. For example, Recht \emph{et al.} in \cite{Recht:2007} showed that nuclear norm relaxation of rank can recover the minimum rank solution if certain property, namely Restricted Isometry Property (RIP), holds for the linear mapping. A family of Iterative Re-weighted Least Squares algorithms which minimize Schatten-p norm, i.e. $\|X\|_{S_p}=\mathbf{Tr}(X^{\text T}X+\gamma I)^{p/2}$, of the matrix as a surrogate for its rank is also introduced in \cite{Mohan:2012}. Singular Value Projection (SVP) algorithm is also guaranteed to recover the low rank solution for affine constraints which satisfy RIP  \cite{Meka:2009}.

\begin{figure}[t]
%\hspace{.5cm}
\centering
    \subfloat[\label{fig:syst_patt}]{%
      \includegraphics[trim = 35mm 10mm 30mm 10mm, clip, width=0.33\textwidth]{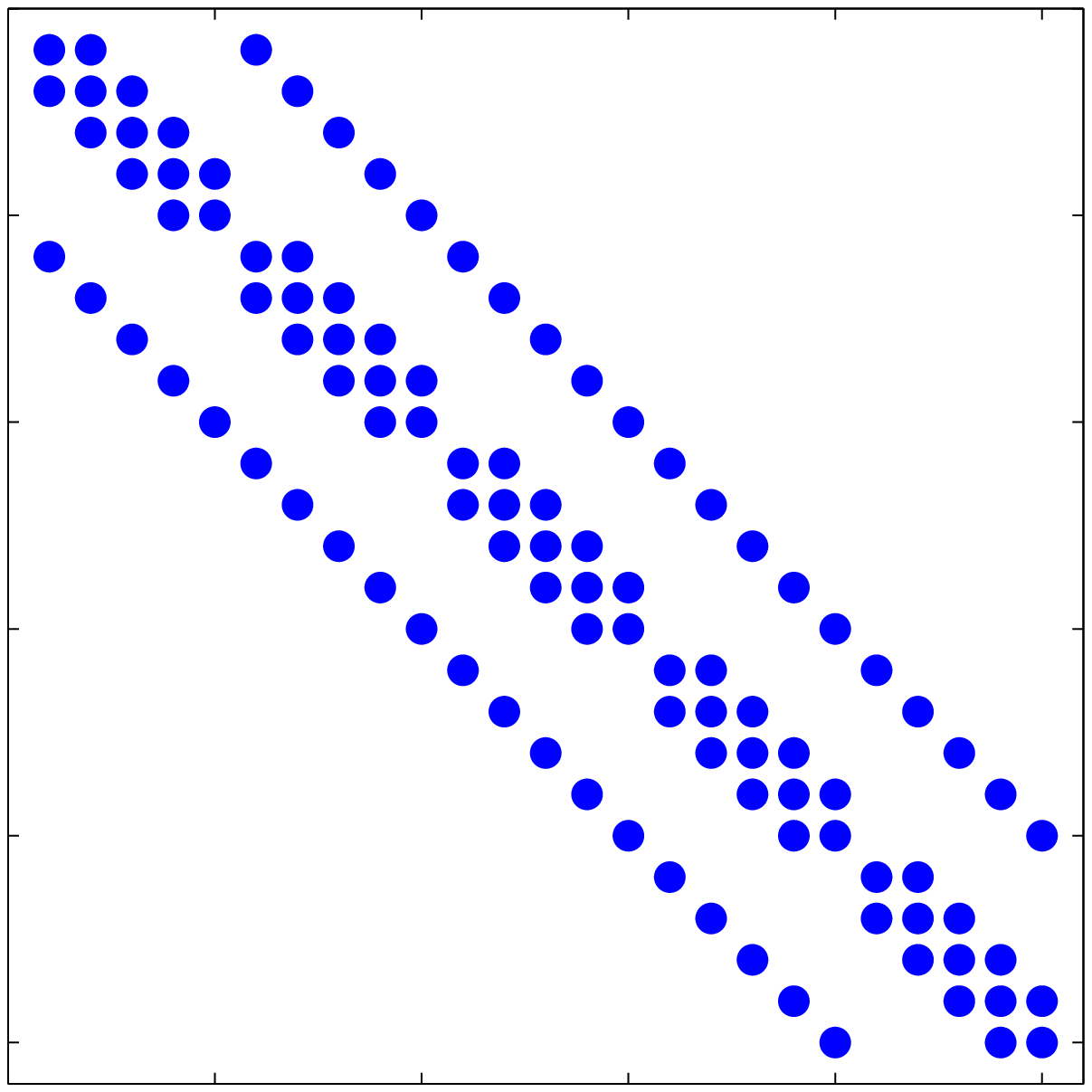} 
    }
    %\hfill
    \subfloat[\label{fig:COFC_patt}]{%
      \includegraphics[trim = 35mm 10mm 30mm 10mm, clip,width=0.33\textwidth]{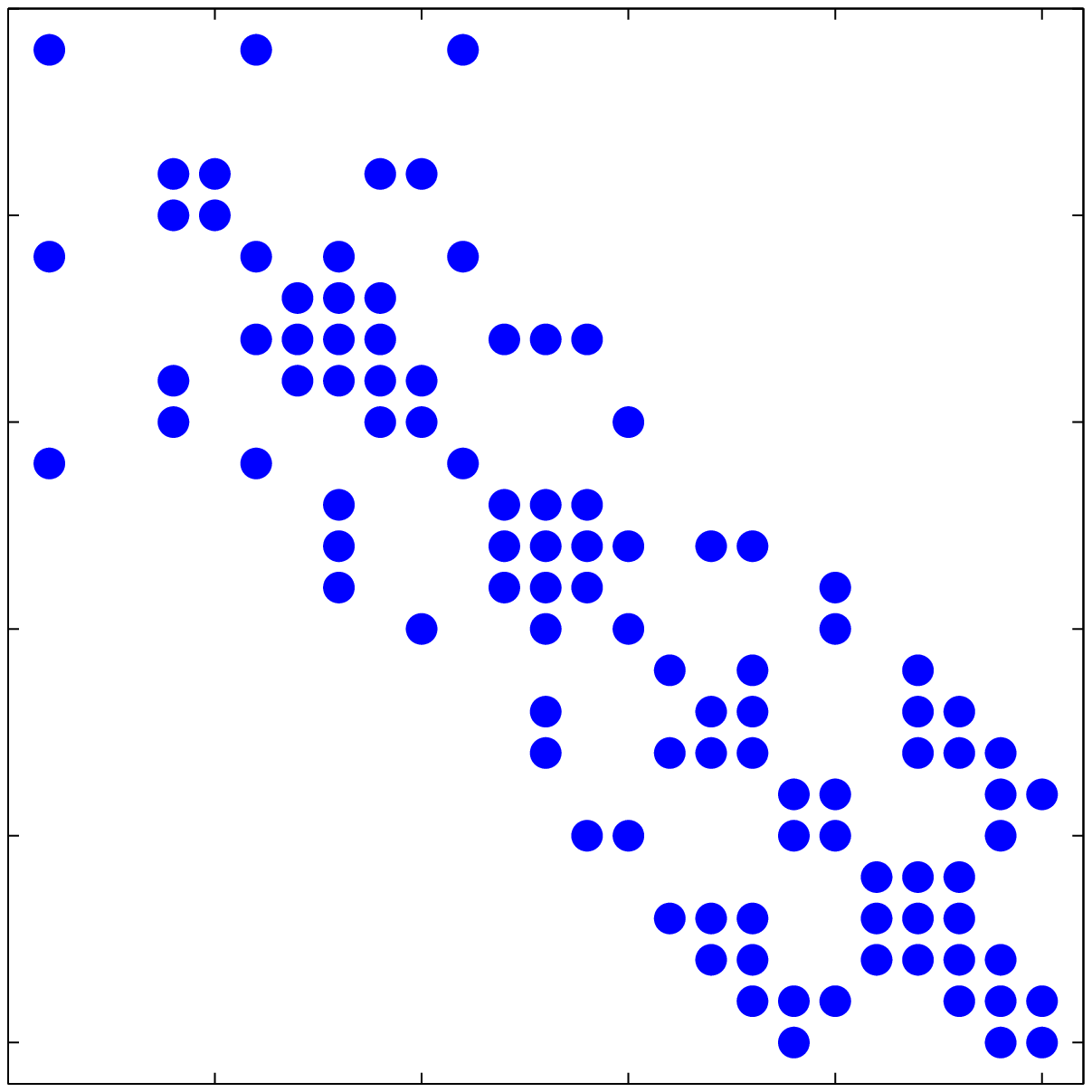}
      }\\
      %\hfill
      \subfloat[\label{fig:COFC_graph}]{%
      \includegraphics[trim = 0mm 10mm 0mm 10mm, clip,width=0.45\textwidth]{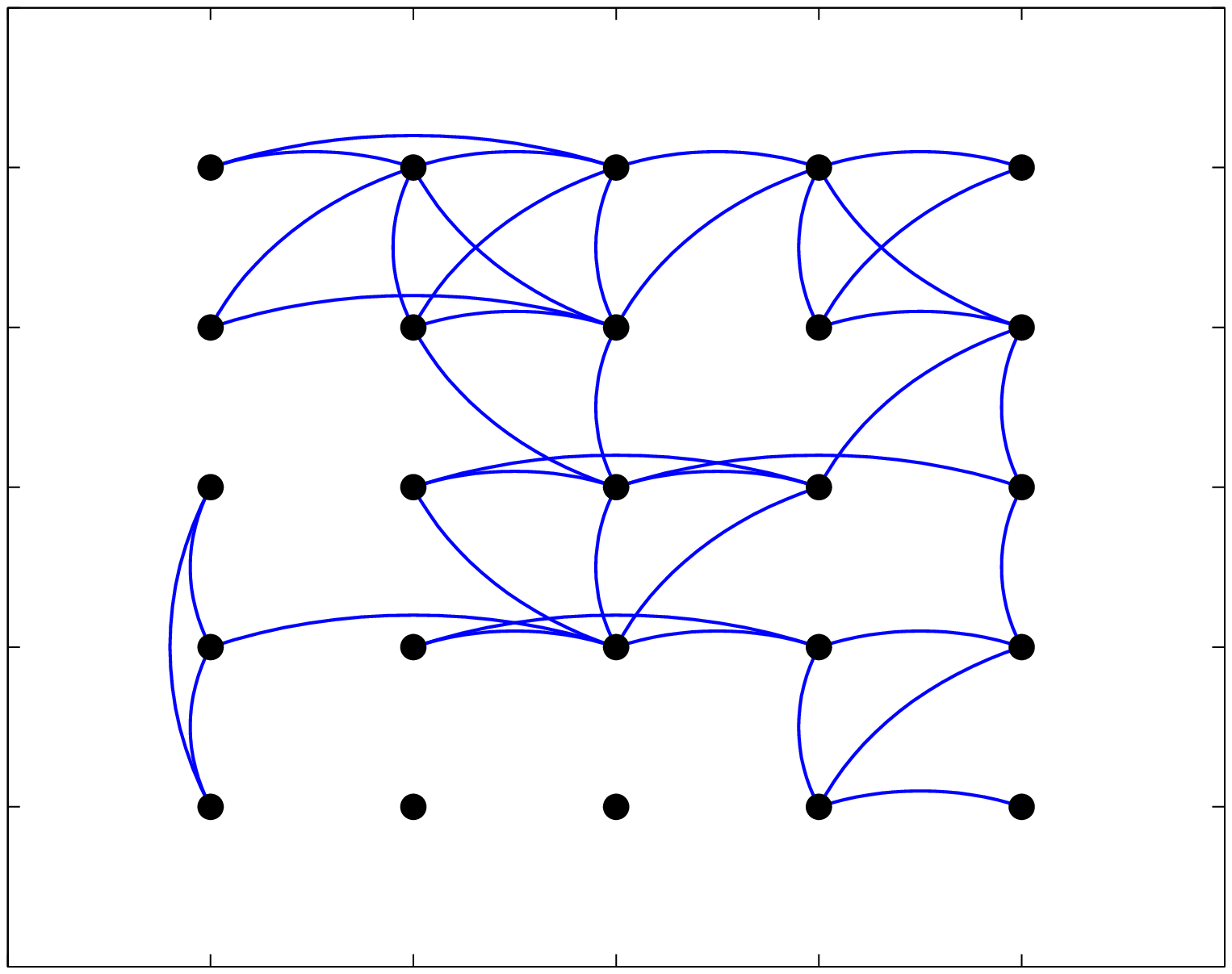}
      }
      \centering
      \caption{ Sparsity pattern of (a) the network system  (b) the optimal sparse feedback controller \{$\lambda=10$, $\rho=100$\}. (c) representation of the underlying graph of the sparse controller.  }
    \label{fig:A}
\end{figure}

\begin{remark}
The discrete-time counterpart of the optimization problem (\ref{eq:rank_const_SPI}) can be formulated as
 \begin{align}\label{eq:rank_const_SPI_disc}
\min_{X_{11},X_{12}, K,N} &\|K\|_0  \tag{P5}\\
\mbox{s.t.}\:\:\: &A^{\text T}X_{11}A+A^{\text T}X_{12}+X_{12}^{\text T}A+X_{22}-X_{11}+N=0, \notag\\
&X_{11}\succ 0,~N \succ 0,\notag\\
&Y^{\text T}=BKC,\notag\\
&K \in \mathcal{K}, \notag\\
&\mathbf{rank}\left[\begin{array}{cc} X_{11}& X_{12}\\X_{12}^{\text T} & X_{22}\\I_{ n}& Y^{\text T} \end{array}\right]=n. \notag
\end{align}
Hence, the results, developed in this section, are applicable to the problem of identifying the sparsest stabilizing controller for discrete-time linear time invariant systems.
\end{remark}
%
%
%
%%%%%%%%%%%%%%%%%%%%%%%%%%%%%%%%%%%%%%%%%%%%%%%%%%%%%%%%%%%%%%%%%%%%%%%%%%%%%%%%%%%%%%%%%%%%%%%%%%%%%%%%%%%%%%%%%%%%%%%%%%%%%%%%%%%%%%%%%%%

\section{Numerical Examples \label{nume-exam-sec}}
\begin{figure}[t!]
%\vspace{-.5cm}
    \subfloat[\label{fig:H2_vs_Gamma}]{%
      \includegraphics[trim = 0mm 5mm 0mm 10mm, clip,width=0.45\textwidth]{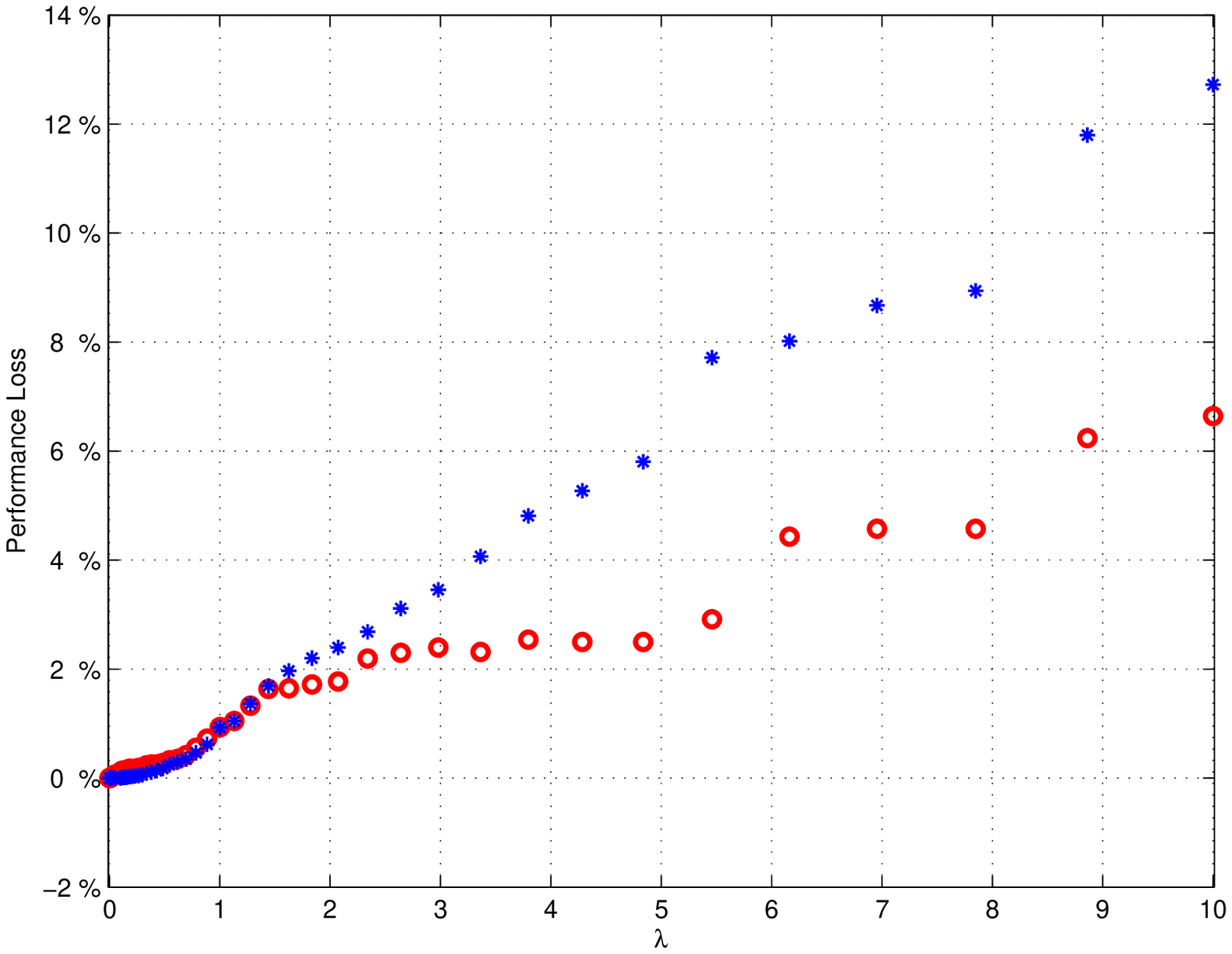}
    }
    %\hfill
    \subfloat[\label{fig:NNZ_vs_Gamma}]{%
      \includegraphics[trim = 0mm 5mm 0mm 10mm, clip, width=0.45\textwidth]{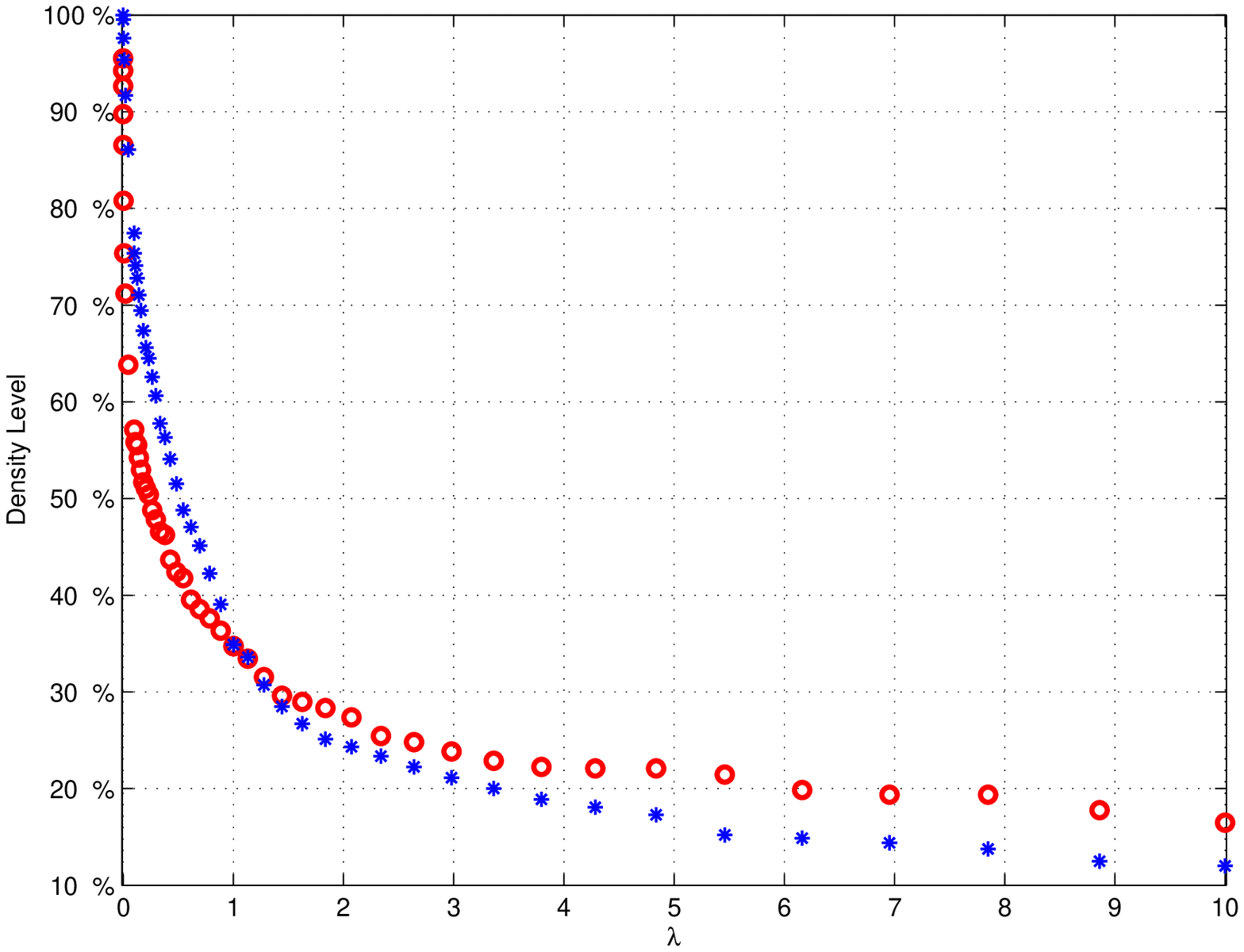}
      }
      \centering
      \caption{ (a) Percentage of optimal quadratic cost degradation relative to the $\text{LQR}$ optimal cost and (b) Density level of the controller gain for different values of $\lambda$, and for the two controller design approaches: SPOFC ($*$) and our proposed method ($\circ$)}
    \label{fig:B}
\end{figure}

In this section, we use several examples to demonstrate how our proposed rank constrained optimization approach can be exploited to solve the optimal sparse output feedback controller design problem considering the input/output constraints.
\subsection{Unstable Lattice Network System}
Here, we illustrate an example in which we design an optimal sparse state feedback controller for an unstable networked system with $25$ states defined on a $5\times 5$ lattice. The entries of its corresponding system matrix are randomly generated scalars drawn from the standard uniform distribution on the open interval $\left(-1,1\right)$, and it is assumed the state performance matrix $Q$ to be an identity matrix, while the control performance weight $R=10I$. Here, we used the traditional $\text{LQR}$ controller as the benchmark to measure the performance of our proposed algorithm. Performing standard $\text{LQR}$ design method, our results show that the optimal cost, for the case of $\text{LQR}$ control design, is $J^*=211.173$. 

Next, we applied Algorithm 1 to design an optimal sparse controller with the parameters values $\lambda=10$ and $\rho=100$, while keeping the performance weights unchanged. It can be observed that the optimal controller cost function increases to $J^*=230.6989$, which is about $9.2\%$ higher, comparing to that of the $\text{LQR}$ design. On the other hand, the number of non-zero entries of the controller gain drops to $97$, i.e. $\|K\|_0=97$. This means a major decrease in the number of non-zero entries of the controller gain. Figures \ref{fig:syst_patt} and \ref{fig:COFC_patt} show the sparsity structure of the system network and the obtained sparse controller. The figures basically visualize the controller matrix by using solid blue circles to represent the non-zero entries of the matrix and leaving the zero entries as blanks. In Figure \ref{fig:COFC_graph} the graph representation of the generated sparse controller is depicted.

Additionally, we present a brief case study that compares our approach with the Sparsity Promoting Optimal Feedback Control (SPOFC) method, proposed in \cite{Lin:2012,Lin:2013}. The SPOFC method essentially solves a different control problem, since it solves the $\mathcal{H}_2$ problem, modified by adding a sparsity promoting penalty function to its cost function and obtain a sub-optimal sparse state feedback controller, while our proposed approach is built upon adjusting the $\text{LQR}$ problem to achieve a sparse {\em output} feedback controller. Moreover, the approach in SPOFC algorithm fails to directly incorporate the norms bounds on the inputs/outputs and the controller predefined structure. Nonetheless, for comparison purposes and demonstrating the comparable performance of our method, we have obtained the MATLAB source code for SPOFC from the website \emph{www.ece.umn.edu/mihailo/sofware/lqrsp}, and applied both our method and SPOFC to design sparse state feedback controllers for the randomly generated system. Fig. \ref{fig:B} depicts the results of the simulations performed using both controller design methods. As predicted, the quadratic cost of the closed loop system increases, as the the parameter $\lambda$ becomes larger. Moreover, increasing the value of this parameters on the system promotes the sparsity level of feedback gain matrix. Figure \ref{fig:B} depicts the effect of the parameter $\lambda$ on the performance of the closed loop system and the number of non-zero entries of the controller gain. In Fig. \ref{fig:H2_vs_Gamma} the $Y$-axis represents percentage of the performance loss, which is defined as $(J^*-J^*_{\text {LQR}})/J^*_{\text{LQR}}$. The density level percentage of the controller gain is also shown in Fig. \ref{fig:NNZ_vs_Gamma} when the parameter $\lambda$ varies from $10^{-3}$ to $10$. 

The simulation results, demonstrated in figure \ref{fig:B}, show that the SPOFC approach compromises the performance for a sparser controller in comparison to the our method. Our proposed method assures less performance loss by obtaining denser feedback controller. The disagreement between the optimal solutions of the two algorithms is mainly due to convergence to different local optima. It should also be noted the optimization parameter $\rho$ plays an important, but different, role in adjusting the convergence properties in both of the methods. Hence, setting the parameter $\rho$ to the same value in both optimizations may not be the most accurate choice for the comparison purposes. Moreover, the choice of the system also affects the design performance of both methods. Overall, our extensive simulation results suggest the comparable performance of both approaches. Considering the fact that our problem formulation and solving procedure, which is completely different from the preceding method, generates roughly the same sparse controller, it can be concluded that the derived sparse controller is likely to be the best we can obtain.  

\subsection{Sub-exponentially Spatially Decaying System}\label{subsec:SSDS}
\begin{figure}[t]
%\vspace{-.5cm}
    \subfloat[\label{fig:performance3D}]{%
      \includegraphics[trim = 0mm 5mm 0mm 10mm, clip,width=0.45\textwidth]{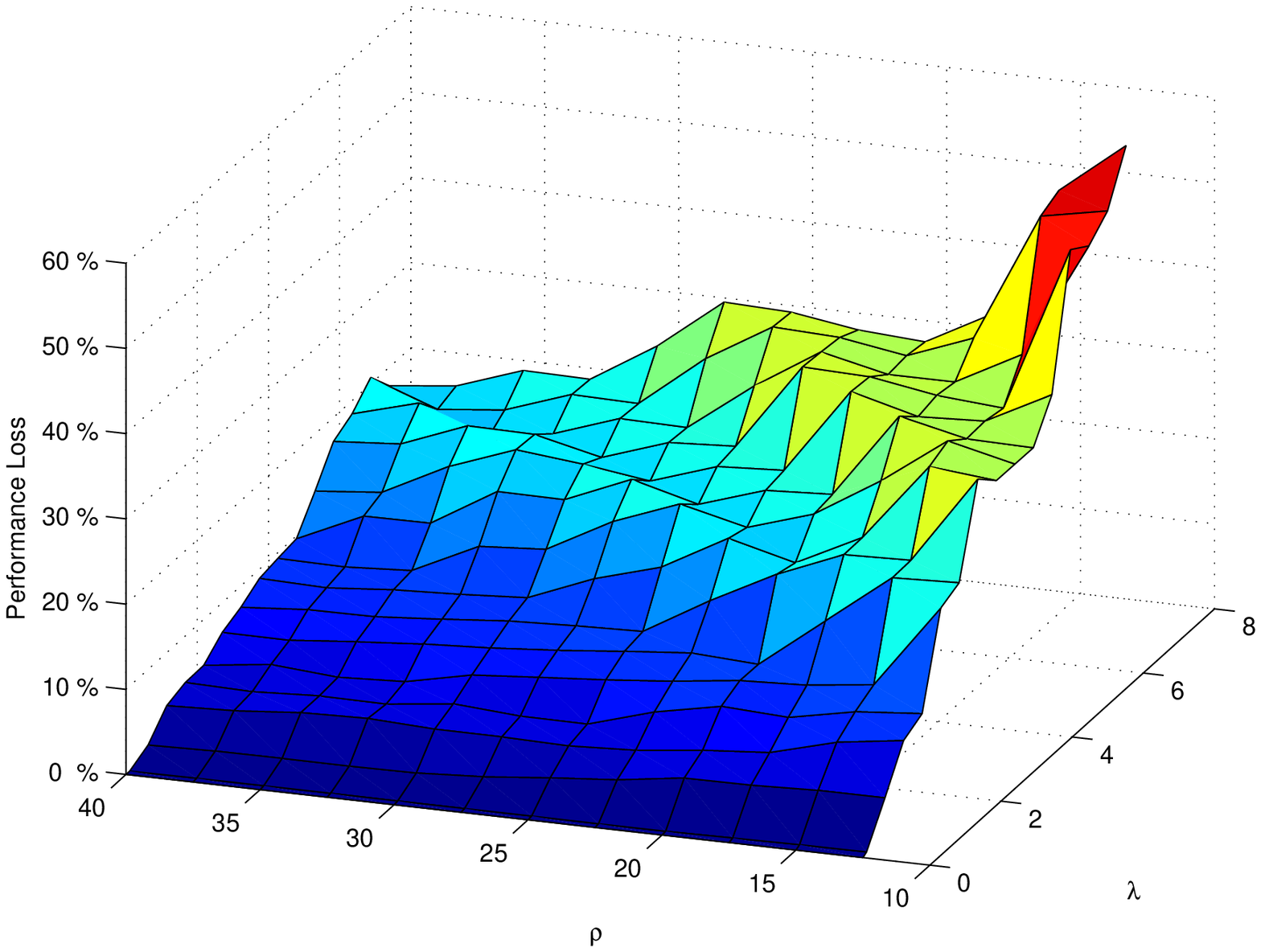}
    }
    %\hfill
    \subfloat[\label{fig:density3D}]{%
      \includegraphics[trim = 0mm 5mm 0mm 10mm, clip, width=0.45\textwidth]{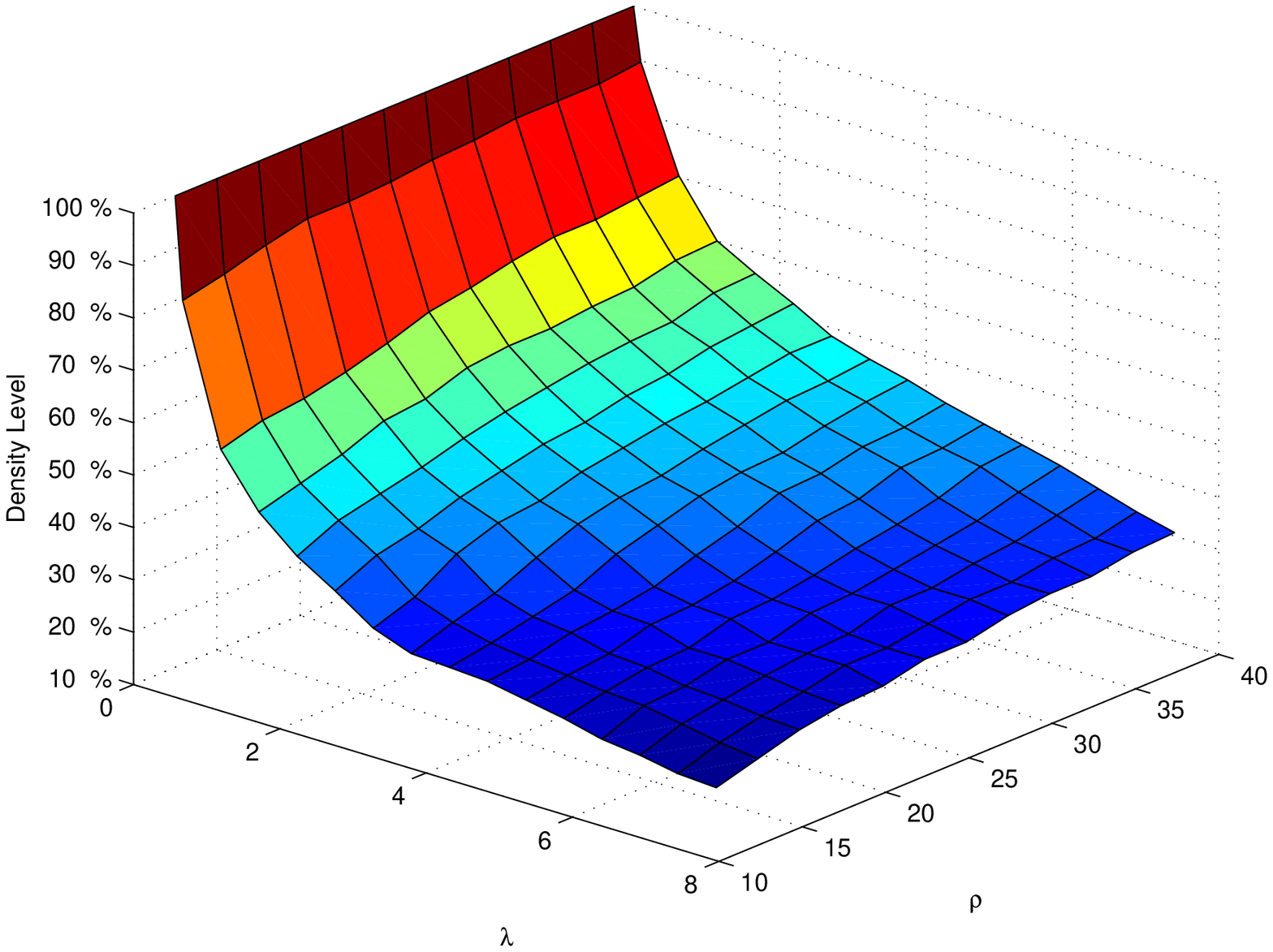}
      }
      \hfill
    \subfloat[\label{fig:DensityPerformance}]{%
      \includegraphics[trim = 0mm 5mm 0mm 10mm, clip, width=0.45\textwidth]{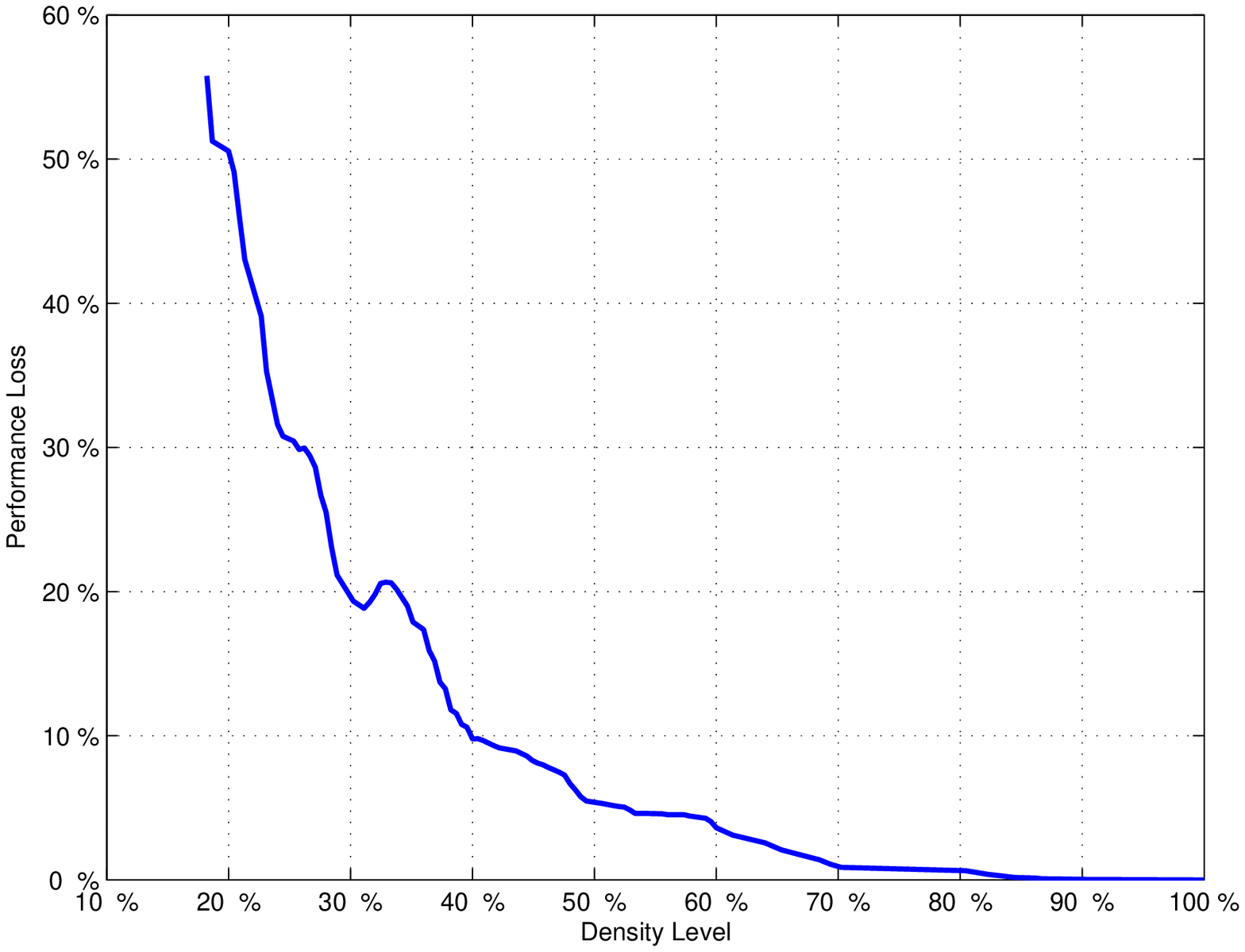}
      }
      \centering
      \caption{The characteristics of the sparse controller designed for a randomly generated spatially decaying system with parameters values \{$C_A=10$, $C_B=2$, $\alpha_A=1$, $\alpha_B=0.4$, $\beta_A=3$, and $\beta_B=0.9$\}. (a) Performance loss vs. $\rho$ and $\lambda$, (b) Density level vs. $\rho$ and $\lambda$ (c) Density level vs. controller performance degradation}
    \label{fig:C}
\end{figure}
To study the effects of parameters $\lambda$ and $\rho$ on the performance of our proposed method, we have run extensive simulations on a randomly generated sub-exponentially spatially decaying system \cite{Motee:2014}. In such systems, it is assumed the entries of the system matrices decay as they get further from the diagonal, thus we define the matrices $A=[a_{ij}]$ and $B=[b_{ij}]$ as 
\begin{align*}
\left\{\begin{array}{l}
a_{ij}=C_A \mathfrak{a}~ e^{-\alpha_A|i-j|^{\beta_A}}\\
b_{ij}=C_B \mathfrak{b}~e^{-\alpha_B|i-j|^{\beta_B}}
\end{array}\right.
\end{align*}
where $\mathfrak{a}$ and $\mathfrak{b}$ are uniformly distributed random variables on the open interval $\left(-1,1\right)$. By employing Algorithm 1 till the rank constraint is satisfied, we have depicted the performance degradation and density level of the generated controllers in figure \ref{fig:C} for different values of $\rho$ and $\lambda$. Although the proposed algorithm has converged for all choices of parameters in this simulation, It seems that the choice of the optimization parameter $\rho$ is needed to be at least one order of magnitude larger than the parameter $\lambda$ in order to guarantee the convergence to a proper sub-optimal minimum. In addition, since the main objective in designing a sparse controller is to obtain a controller with minimum number of nonzero entries and lowest performance decline, we have also presented the plot of the lowest performance loss obtained for particular values of density level in figure \ref{fig:DensityPerformance}. As expected, it can be observed the performance loss grows as the sprasity level of the controller increases.

\subsection{Optimal Sparse Controller with Upper Bound Imposed on the Control Input Norm}
\begin{figure}[t]
%\hspace{.5cm}
\centering
    \subfloat[\label{fig:struc_boundless}]{%
      \includegraphics[trim = 35mm 10mm 30mm 10mm, clip, width=0.33\textwidth]{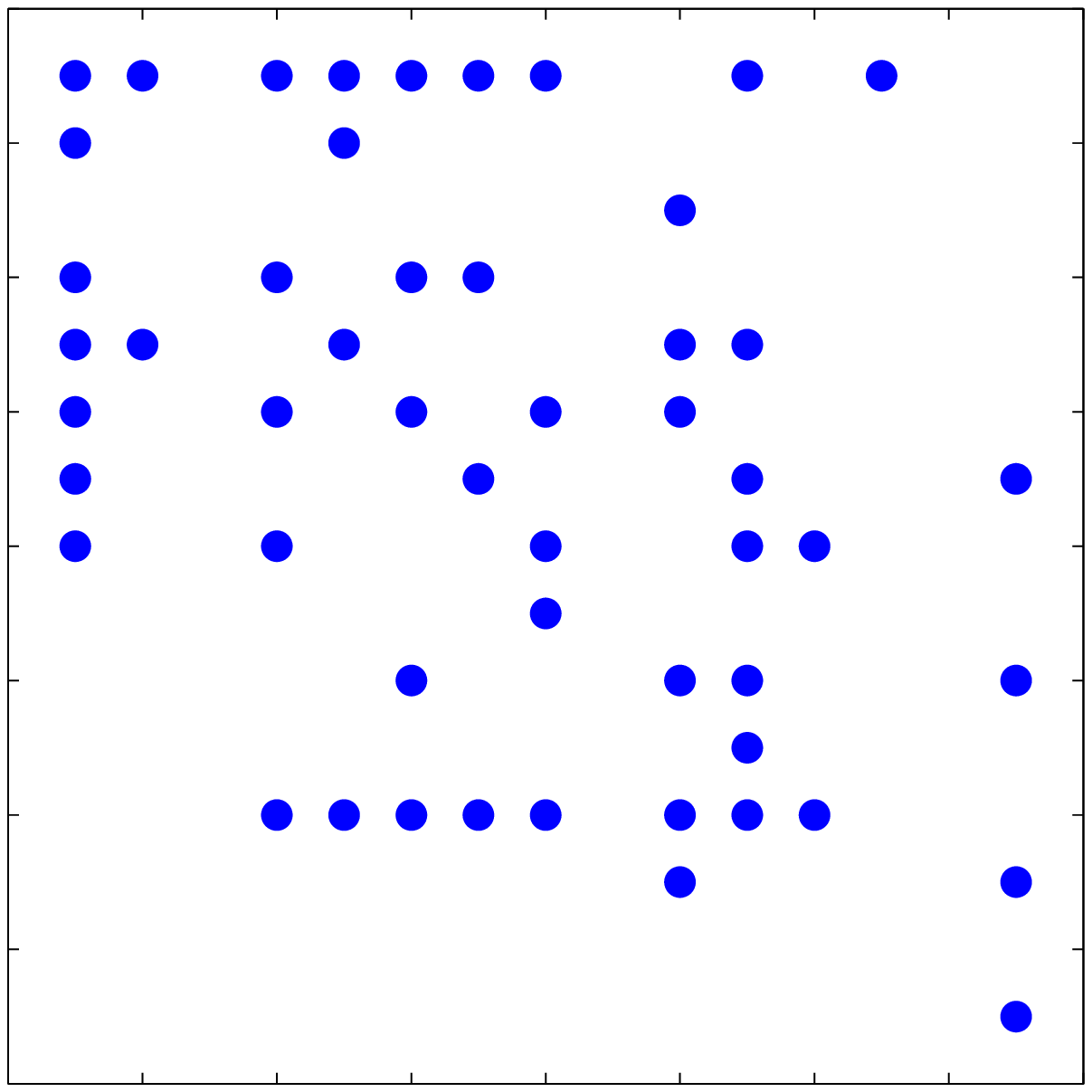} 
    }
    %\hfill
    \subfloat[\label{fig:struc_bound}]{%
      \includegraphics[trim = 35mm 10mm 30mm 10mm, clip,width=0.33\textwidth]{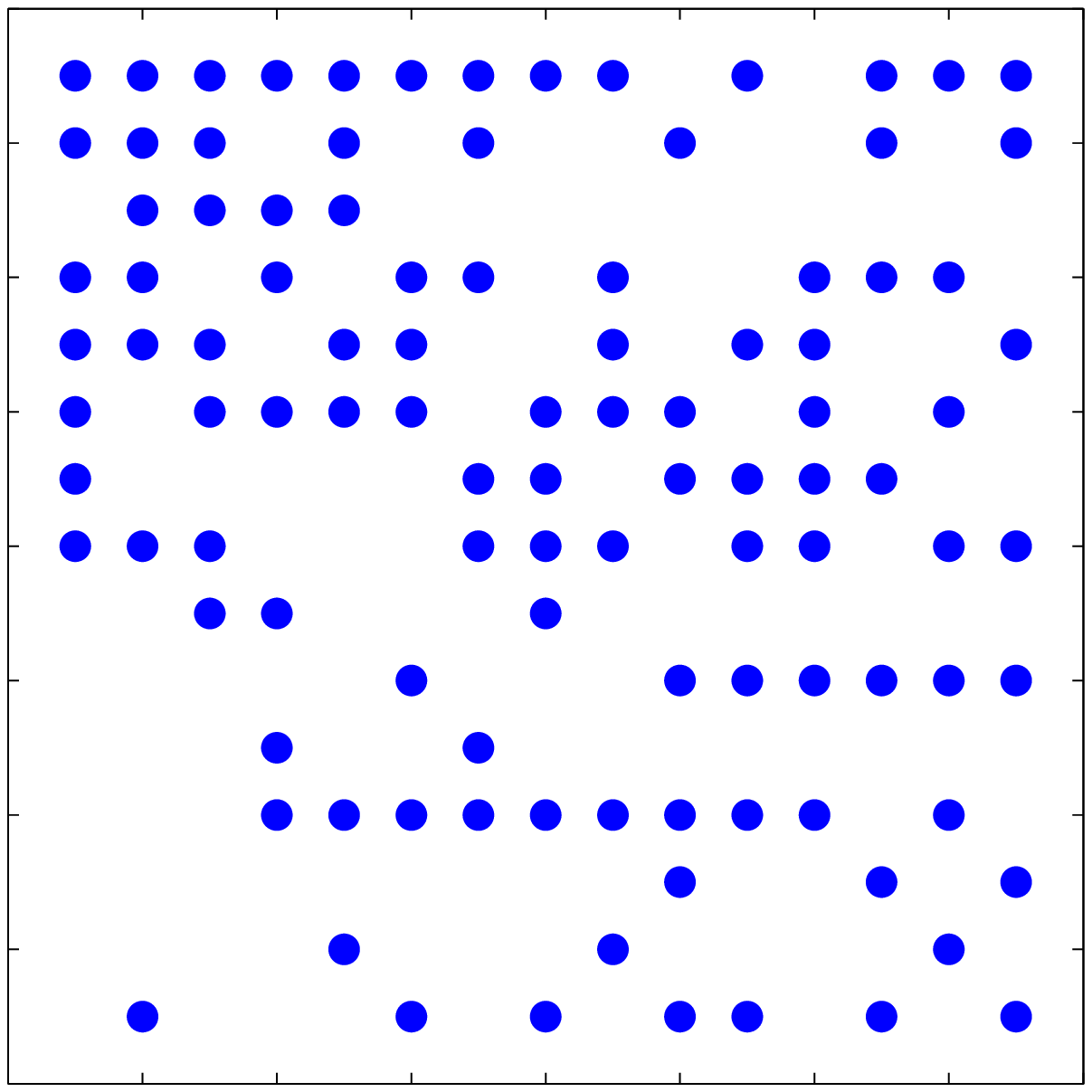}
      }
      \centering
      \caption{ Sparsity pattern of (a) optimal sparse feedback control with no bound on the control input  (b) the optimal sparse feedback controller with upper bound imposed on the system norm of the control input. Design Parameters for both figures are $Q=I$, $R=10I$, $\lambda=10$ and $\rho=100$. }
    \label{fig:D}
\end{figure}
In this example, we illustrate the effect of bounding the norm of the control input on the sparsity of the controller matrix. Considering a randomly generated $16 \times 16$ sub-exponential spatially decaying system, with the same parameter values used in section \ref{subsec:SSDS}, we first designed a sparse controller with no constraint on the control input. Our results show that the controllers number of nonzero entries and its performance loss, with respect to the cost of the $\text{LQR}$ controller which is $639.1912$, are $55$ and $9.3\%$ respectively. It is also observed that for the generated controller, we have $\|u\|_{L^{2}_{\infty}(\mathbb{R}^m)}=228.66$.

We then redesigned the controller, using the \emph{re-weighted $\ell_1$ minimization} method, by containing its control input norm in the interval $\left[0,200\right]$, and obtained controller has the following characteristics: $\|K\|_0=105$ and $J=737.16$. Although we bounded the control input norm to an approximately $10\%$ lower value, the obtained controller demonstrates $50\%$ less sparse pattern and  $6\%$ higher performance loss. The simulations results, depicted in figure \ref{fig:D}, not only verifies the capability of our method to incorporate bounds on the control input, as well as the system output, but also reveals the adverse impact of sparsifying the controller matrix on the control input norm.
%
%
%
%%%%%%%%%%%%%%%%%%%%%%%%%%%%%%%%%%%%%%%%%%%%%%%%%%%%%%%%%%%%%%%%%%%%%%%%%%%%%%%%%%%%%%%%%%%%%%%%%%%%%%%%%%%%%%%%%%%%%%%%%%%%%%%%%%%%%%%%%%%
\section{CONCLUSIONS \label{conc-futu-sec}}
In this paper, We have proposed a new framework for optimal sparse output feedback control design, which is capable of incorporating structural constraints on the feedback gain matrix as well as norm bounds on the inputs/outputs of the system. We have shown that problem can be converted to a rank constrained optimization problem with no other non-convex constraints. Using the proposed formulation, we have presented an optimization problem which yields an upper bound for the optimal value of the optimal sparse state feedback control problem. Exploiting the relaxation the $\ell_0$-norm with the $\ell_1$-norm, We have also expressed that local optimum of the relaxed optimization problem, in its general form, can be obtained by performing ADMM algorithm, which is, in essence, iteratively solving the relaxed problem and projecting its solution to the space of matrices with rank $n$. For the special case, where the objective is merely sparsity pattern recognition of the controller gain, we have demonstrated that the problem can be reduced to an Affine Rank Minimization. The simulation results are also provided to illustrate the utility and performance of our proposed approach. As compared to the results of \cite{Lin:2013}, our results show that while our proposed method has the advantage of performing the output feedback control design restricted by various forms of nonlinear constraints, the performance of our approach is on a par with theirs when applied to the regular sparse state feedback controller design problem.
%%%%%%%%%%%%%%%%%%%%%%%%%%%%%%%%%%%%%%%%%%%%%%%%%%%%%%%%%%%%%%%%%%%%%%%%%%%%%%%%%%%%%%%%%%%%%%%%%%%%%%%%%%%%%%%%%%%%%%%%%%%%%%%%%%%%%%%%%%%

%\section{ACKNOWLEDGMENTS}
%
%The authors gratefully acknowledge the contribution of National Research Organization and reviewers' comments.
%
%
%%%%%%%%%%%%%%%%%%%%%%%%%%%%%%%%%%%%%%%%%%%%%%%%%%%%%%%%%%%%%%%%%%%%%%%%%%%%%%%%%
%
%References are important to the reader; therefore, each citation must be complete and correct. If at all possible, references should be commonly available publications.

\bibliographystyle{IEEEtran}
\bibliography{Arasbib}

\begin{thebibliography}{10}
\providecommand{\url}[1]{#1}
\csname url@samestyle\endcsname
\providecommand{\newblock}{\relax}
\providecommand{\bibinfo}[2]{#2}
\providecommand{\BIBentrySTDinterwordspacing}{\spaceskip=0pt\relax}
\providecommand{\BIBentryALTinterwordstretchfactor}{4}
\providecommand{\BIBentryALTinterwordspacing}{\spaceskip=\fontdimen2\font plus
\BIBentryALTinterwordstretchfactor\fontdimen3\font minus
  \fontdimen4\font\relax}
\providecommand{\BIBforeignlanguage}[2]{{%
\expandafter\ifx\csname l@#1\endcsname\relax
\typeout{** WARNING: IEEEtran.bst: No hyphenation pattern has been}%
\typeout{** loaded for the language `#1'. Using the pattern for}%
\typeout{** the default language instead.}%
\else
\language=\csname l@#1\endcsname
\fi
#2}}
\providecommand{\BIBdecl}{\relax}
\BIBdecl

\bibitem{Blondel:1997}
V.~Blondel and J.~N. Tsitsiklis, ``{NP}-hardness of some linear control design
  problems,'' \emph{SIAM Journal on Control and Optimization}, vol.~35, pp.
  2118 -- 2127, 1997.

\bibitem{Schuler:2011}
S.~Schuler, P.~Li, J.~Lam, and F.~Allg\"{o}wer, ``Design of structured dynamic
  output-feedback controllers for interconnected systems,'' \emph{International
  Journal of Control}, vol.~84, no.~12, pp. 2081--2091, 2011.

\bibitem{Polyak:2013}
B.~Polyak, M.~Khlebnikov, and P.~Shcherbakov, ``An {LMI} approach to structured
  sparse feedback design in linear control systems,'' in \emph{Proceeding of
  the 2013 European Control Conference}, July 2013, pp. 833--838.

\bibitem{Motee:2009}
N.~Motee and A.~Jadbabaie, ``Distributed multi-parametric quadratic
  programming,'' \emph{IEEE Transactions on Automatic Control}, vol.~54,
  no.~10, pp. 2279--2289, Oct 2009.

\bibitem{Lin:2011}
F.~Lin, M.~Fardad, and M.~Jovanov\'ic, ``Augmented {L}agrangian approach to
  design of structured optimal state feedback gains,'' \emph{IEEE Transactions
  on Automatic Control}, vol.~56, no.~12, pp. 2923--2929, Dec 2011.

\bibitem{Krishna:2013}
\BIBentryALTinterwordspacing
K.~Dvijotham, E.~Todorov, and M.~Fazel, ``Convex structured controller
  design,'' \emph{CoRR}, vol. abs/1309.7731, 2013. [Online]. Available:
  \url{http://arxiv.org/abs/1309.7731}
\BIBentrySTDinterwordspacing

\bibitem{Bamieh:2002}
B.~Bamieh, F.~Paganini, and M.~A. Dahleh, ``Distributed control of
  spatially-invariant systems,'' \emph{IEEE Transactions on Automatic Control},
  vol.~47, pp. 1091--1107, 2002.

\bibitem{Bamieh:2005}
B.~Bamieh and P.~Voulgaris, ``A convex characterization of distributed control
  problems in spatially invariant systems with communication constraints,''
  \emph{Systems and Control Letters}, vol.~54, pp. 575--583, 2005.

\bibitem{Motee:2008}
N.~Motee and A.~Jadbabaie, ``Optimal control of spatially distributed
  systems,'' \emph{IEEE Transactions on Automatic Control}, vol.~53, no.~7, pp.
  1616--1629, Aug 2008.

\bibitem{Lavaei:2013}
J.~Lavaei, ``Optimal decentralized control problem as a rank-constrained
  optimization,'' in \emph{2013 51st Annual Allerton Conference on
  Communication, Control, and Computing (Allerton)}, Oct 2013, pp. 39--45.

\bibitem{Fazelnia:2014}
G.~Fazelnia, R.~Madani, A.~Kalbat, and J.~Lavaei, ``Convex relaxation for
  optimal distributed control problem -- {Part I}: {T}ime-domain formulation,''
  \emph{optimization}, vol.~16, p.~17.

\bibitem{Madani:2014}
R.~Madani, G.~Fazelnia, S.~Sojoudi, and J.~Lavaei, ``Low-rank solutions of
  matrix inequalities with applications to polynomial optimization and matrix
  completion problems,'' \emph{Proceedings of the 2014 IEEE Conference on
  Decision and Control}, 2014.

\bibitem{Wang:2014}
Y.~S. Wang, N.~Matni, and J.~C. Doyle, ``Localized {LQR} optimal control,''
  \emph{CoRR}, vol. abs/1409.6404, 2014.

\bibitem{Wang:2015}
Y.-S. Wang and N.~Matni, ``Localized lqg optimal control for large-scale
  systems,'' \emph{submitted to 2015 54th IEEE Conference on Decision and
  Control (CDC)}, 2015.

\bibitem{Rotkowitz:2006}
M.~Rotkowitz and S.~Lall, ``A characterization of convex problems in
  decentralized control,'' \emph{IEEE Transactions on Automatic Control},
  vol.~51, no.~2, pp. 274--286, Feb 2006.

\bibitem{Rotkowitz:2011}
M.~Rotkowitz, ``Parametrization of stabilizing controllers subject to subspace
  constraints,'' in \emph{American Control Conference (ACC), 2011}, June 2011,
  pp. 5370--5375.

\bibitem{XinQi:2004}
X.~Qi, M.~Salapaka, P.~Voulgaris, and M.~Khammash, ``Structured optimal and
  robust control with multiple criteria: {A} convex solution,'' \emph{IEEE
  Transactions on Automatic Control}, vol.~49, no.~10, pp. 1623--1640, Oct
  2004.

\bibitem{Donoho:2006}
D.~Donoho, ``Compressed sensing,'' \emph{IEEE Transactions on Information
  Theory}, vol.~52, no.~4, pp. 1289--1306, April 2006.

\bibitem{DeVore:2007}
R.~A. DeVore, ``Deterministic constructions of compressed sensing matrices,''
  \emph{Journal of Complexity}, vol.~23, no. 4–6, pp. 918 -- 925, 2007.

\bibitem{Zahedi:2013}
R.~Zahedi, L.~Krakow, E.~Chong, and A.~Pezeshki, ``Adaptive estimation of
  time-varying sparse signals,'' \emph{IEEE Access}, vol.~1, pp. 449--464,
  2013.

\bibitem{Zahedi:2012}
R.~Zahedi, A.~Pezeshki, and E.~K. Chong, ``Measurement design for detecting
  sparse signals,'' \emph{Physical Communication}, vol.~5, no.~2, pp. 64 -- 75,
  2012, compressive Sensing in Communications.

\bibitem{Davenport:2010}
M.~Davenport, P.~Boufounos, M.~Wakin, and R.~Baraniuk, ``Signal processing with
  compressive measurements,'' \emph{IEEE Journal of Selected Topics in Signal
  Processing}, vol.~4, no.~2, pp. 445--460, April 2010.

\bibitem{Candes:2004}
E.~Candes and T.~Tao, ``Decoding by linear programming,'' \emph{Information
  Theory, IEEE Transactions on}, vol.~51, no.~12, pp. 4203--4215, Dec 2005.

\bibitem{Candes:2005}
E.~J. Candès, J.~K. Romberg, and T.~Tao, ``Stable signal recovery from
  incomplete and inaccurate measurements,'' \emph{Communications on Pure and
  Applied Mathematics}, vol.~59, no.~8, pp. 1207--1223, 2006.

\bibitem{Candes:2008}
E.~Candes, M.~Wakin, and S.~Boyd, ``Enhancing sparsity by reweighted $\ell_1$
  minimization,'' \emph{Journal of Fourier Analysis and Applications}, vol.~14,
  pp. 877--905, July 2008.

\bibitem{Lin:2013}
F.~Lin, M.~Fardad, and M.~R. Jovanovi\'c, ``Design of optimal sparse feedback
  gains via the {A}lternating {D}irection {M}ethod of {M}ultipliers,''
  \emph{IEEE Transactions on Automatic Control}, vol.~58, no.~9, pp.
  2426--2431, 2013.

\bibitem{Schuler:2011c}
S.~Schuler, C.~Ebembauer, and F.~Allg\"{o}wer, ``$\ell_0$-system gain and
  $\ell_1$-optimal control,'' in \emph{Preprints of the 18th IFAC World
  Congress}, August 2011, pp. 9230--9235.

\bibitem{Arastoo:2012}
V.~V. Kulkarni, R.~Arastoo, A.~Bhat, K.~Subramanian, M.~V. Kothare, and M.~C.
  Riedel, ``Gene regulatory network modeling using literature curated and high
  throughput data,'' \emph{Systems and Synthetic Biology}, vol.~6, no. 3-4, pp.
  69--77, December 2012.

\bibitem{Chartrand:2007}
R.~Chartrand, ``Exact reconstruction of sparse signals via nonconvex
  minimization,'' \emph{Signal Processing Letters, IEEE}, vol.~14, no.~10, pp.
  707--710, Oct 2007.

\bibitem{Lai:2011}
M.~Lai and J.~Wang, ``An unconstrained $\ell_q$ minimization with $0<q\leq1$
  for sparse solution of underdetermined linear systems,'' \emph{SIAM Journal
  on Optimization}, vol.~21, no.~1, pp. 82--101, 2011.

\bibitem{Motee:2013}
N.~Motee and Q.~Sun, ``Measuring sparsity in spatially interconnected
  systems,'' in \emph{2013 IEEE 52nd Annual Conference on Decision and
  Control}, Dec 2013, pp. 1520--1525.

\bibitem{Motee:2014}
------, ``Sparsity and spatial localization measures for spatially distributed
  systems,'' \emph{arXiv preprint arXiv:1402.4148}, 2014.

\bibitem{Motee:2014ACC}
------, ``Sparsity measures for spatially decaying systems,'' in
  \emph{Proceedings of the 2014 American Control Conference}, June 2014, pp.
  5459--5464.

\bibitem{Recht:2010}
B.~Recht, W.~Xu, and B.~Hassibi, ``Necesary and sufficient conditions for
  success of the nuclear norm heuristic for rank minimization,'' in
  \emph{Proceedings of the 47$^{\rm th}$ IEEE Conference on Decision and
  Control}.\hskip 1em plus 0.5em minus 0.4em\relax Cancun, Mexico, December
  2008, pp. 180--185.

\bibitem{Recht:2011}
------, ``Null space conditions and thresholds for rank minimization,''
  \emph{Mathematical Programming}, vol. 127, pp. 175--211, July 2011.

\bibitem{Mesbahi:1997}
M.~Mesbahi and G.~P. Papavassilopolous, ``On the rank minimization problem over
  a positive semidefinite {L}inear {M}atrix {I}nequality,'' \emph{IEEE
  Transacation on Automatic Control}, vol.~42, pp. 1239--243, February 1997],.

\bibitem{Recht:2007}
B.~Recht, M.~Fazel, and P.~A. Parillo, ``Guaranteed minimum rank solutions to
  {L}inear {M}atrix {E}quations via nuclear norm minimization,'' \emph{SIAM
  Review}, vol.~52, pp. 471--501, July 2007.

\bibitem{Feron:1992}
E.~Feron, V.~Balakrishnan, S.~Boyd, and L.~El~Ghaoui, ``Numerical methods for
  $\mathcal{H}_2$ related problems,'' in \emph{Proceedings of the 1992 American
  Control Conference}, June 1992, pp. 2921--2922.

\bibitem{Dattoro:2005}
J.~Dattoro, \emph{Convex Optimization \& Euclidean Distance Geometry}.\hskip
  1em plus 0.5em minus 0.4em\relax Meboo Publishing, 2005.

\bibitem{Delgado:2014}
R.~A. Delgado, J.~C. Ag{\"u}ero, and G.~C. Goodwin, ``A rank-constrained
  optimization approach: Application to factor analysis,'' \emph{19th IFAC
  World Congress}, 2014.

\bibitem{Gorski:2007}
``Biconvex sets and optimization with biconvex functions: a survey and
  extensions,'' \emph{Mathematical Methods of Operations Research}, vol.~66,
  no.~3, 2007.

\bibitem{Boyd:1994}
S.~Boyd, L.~El~Ghaoui, E.~Feron, and V.~Balakrishnan, \emph{Linear matrix
  inequalities in system and control theory}.\hskip 1em plus 0.5em minus
  0.4em\relax SIAM, 1994, vol.~15.

\bibitem{Candes:2006}
E.~J. Candes, ``Compressive sampling,'' in \emph{International Congress of
  Mathematicians}, vol. III.\hskip 1em plus 0.5em minus 0.4em\relax Zurich:
  Eur.~Math.~Soc., December 2006, pp. 1433--1452.

\bibitem{Yu:2011}
H.~Yu and V.~K.~N. Lau, ``Rank-constrained {S}chur-convex optimization with
  multiple trace/log-det constraints,'' \emph{IEEE Transacation on Signal
  Processing}, vol.~59, pp. 304--314, 2011.

\bibitem{Goemans:1995}
M.~X. Goemans and D.~Williamson, ``Improved approximation algorithms for
  maximum cut and satisfiability problems using semidefinite programming,''
  \emph{J. ACM}, vol.~42, no.~6, pp. 1115--1145, Nov. 1995.

\bibitem{Orsi:2006}
``A newton-like method for solving rank constrained linear matrix
  inequalities,'' \emph{Automatica}, vol.~42, no.~11, pp. 1875 -- 1882, 2006.

\bibitem{Glowinski:1975}
M.~A. Glowinski, R., ``\BIBforeignlanguage{fre}{Sur l'approximation, par
  éléments finis d'ordre un, et la résolution, par pénalisation-dualité
  d'une classe de problèmes de dirichlet non linéaires},''
  \emph{\BIBforeignlanguage{fre}{ESAIM: Mathematical Modelling and Numerical
  Analysis - Modélisation Mathématique et Analyse Numérique}}, vol.~9,
  no.~R2, pp. 41--76, 1975.

\bibitem{Gabay:1976}
D.~Gabay and B.~Mercier, ``A dual algorithm for the solution of nonlinear
  variational problems via finite element approximation,'' \emph{Computers \&
  Mathematics with Applications}, vol.~2, no.~1, pp. 17 -- 40, 1976.

\bibitem{Boyd:2011}
S.~Boyd, N.~Parikh, E.~Chu, B.~Peleato, and J.~Eckstein, ``Distributed
  optimization and statistical learning via the {A}lternating {D}irection
  {M}ethod of {M}ultipliers,'' \emph{Foundations and Trends in Machine
  Learning}, vol.~3, pp. 1--124, 2011.

\bibitem{Oymak:2012}
S.~Oymak, A.~Jalali, M.~Fazel, Y.~C. Eldar, and B.~Hassibi, ``Simultaneously
  structured models with application to sparse and low-rank matrices,''
  \emph{CoRR}, vol. abs/1212.3753, 2012.

\bibitem{Chen:2014}
Y.~Chen, Y.~Chi, and A.~Goldsmith, ``Estimation of simultaneously structured
  covariance matrices from quadratic measurements,'' in \emph{Acoustics, Speech
  and Signal Processing (ICASSP), 2014 IEEE International Conference on}, May
  2014, pp. 7669--7673.

\bibitem{Mohan:2012}
K.~Mohan and M.~Fazel, ``Iterative learning algorithms for matrix rank
  minimization,'' \emph{Journal of Machine Learning Research}, vol.~13.

\bibitem{Meka:2009}
R.~Meka, P.~Jain, and I.~S. Dhillon, ``Guaranteed rank minimization via
  singular value projection,'' \emph{CoRR}, vol. abs/0909.5457, 2009.

\bibitem{Lin:2012}
F.~Lin, M.~Fardad, and M.~R. Jovanovi\'c, ``Sparse feedback synthesis via the
  {A}lternating {D}irection {M}ethod of {M}ultipliers,'' in \emph{Proceedings
  of the 2012 American Control Conference}, 2012, pp. 4765--4770.

\end{thebibliography}

\addtolength{\textheight}{-12cm}   % This command serves to balance the column lengths

\end{document}